\documentclass{amsart}

\newtheorem{theorem}{Theorem}[section]
\newtheorem{lemma}[theorem]{Lemma}
\newtheorem{proposition}[theorem]{Proposition}
\newtheorem{corollary}[theorem]{Corollary}

\theoremstyle{definition}

\newtheorem{example}[theorem]{Example}

\newtheorem{remark}[theorem]{Remark}
\newtheorem{algorithm}[theorem]{Algorithm}

\usepackage{amscd,amssymb}

\begin{document}

\title[Computing with rational symmetric functions]
{Computing with rational symmetric functions \\
and applications to invariant theory\\
and PI-algebras}

\author[F. Benanti, S. Boumova, V. Drensky, G. Genov, and P. Koev]
{Francesca Benanti, Silvia Boumova, Vesselin Drensky,\\
Georgi K. Genov, and Plamen Koev}

\address{Francesca Benanti: Dipartimento di Matematica ed Applicazioni, Universit\`a di Palermo,
Via Archirafi 34, 90123 Palermo, Italy}
\email{fbenanti@math.unipa.it}
\address{Silvia Boumova: Higher School of Civil Engineering ``Lyuben Karavelov'',
175 Suhodolska Str., 1373 Sofia, Bulgaria, and
Institute of Mathematics and Informatics,
Bulgarian Academy of Sciences,
1113 Sofia, Bulgaria}
\email{silvi@math.bas.bg}
\address{Vesselin Drensky and Georgi K. Genov: Institute of Mathematics and Informatics,
Bulgarian Academy of Sciences,
1113 Sofia, Bulgaria}
\email{drensky@math.bas.bg, guenovg@mail.bg}
\address{Plamen Koev: Department of Mathematics,
San Jos\'e State University, San Jose, CA 95192-0103, U.S.A.}
\email{plamen.koev@sjsu.edu}

\thanks
{The research of the first named author was partially supported by INdAM}
\thanks
{The research of the fifth named author was partially supported by NSF Grant DMS-1016086.}


\maketitle

\centerline{\it Dedicated to Yuri Bahturin on the occasion of his 65th birthday}

\begin{abstract}
Let $K$ be a field of any characteristic. Let the formal power series
\[
f(x_1,\ldots,x_d)=\sum \alpha_nx_1^{n_1}\cdots x_d^{n_d}
=\sum m(\lambda)S_{\lambda}(x_1,\ldots,x_d),\quad \alpha_n,m(\lambda)\in K,
\]
be a symmetric function decomposed as a series of Schur functions.
When $f$ is a rational function whose
denominator is a product of binomials of the form $1-x_1^{a_1}\cdots x_d^{a_d}$,
we use a classical combinatorial method of Elliott of 1903
further developed in the $\Omega$-calculus (or Partition Analysis)
of MacMahon in 1916 to compute the generating function
\[
M(f;x_1,\ldots,x_d)=\sum m(\lambda)x_1^{\lambda_1}\cdots x_d^{\lambda_d},\quad \lambda=(\lambda_1,\ldots,\lambda_d).
\]
$M$ is a rational function with denominator of a similar form as $f$.
We apply the method to several problems on symmetric algebras, as well as problems in classical
invariant theory, algebras with polynomial identities, and noncommutative invariant theory.
\end{abstract}


\section*{Introduction}

Let $K$ be a field of any characteristic and let $K[[X]]^{S_d}$
be the subalgebra of the symmetric functions in the algebra of formal power series
$K[[X]]=K[[x_1,\ldots,x_d]]$ in the set of variables $X=\{x_1,\ldots,x_d\}$.
We study series $f(X)\in K[[X]]^{S_d}$
which can be represented as rational functions whose denominators are products
of binomials of the form $1-X^a=1-x_1^{a_1}\cdots x_d^{a_d}$.
Following Berele \cite{B2}, we call such functions {\it nice rational symmetric functions}.
Those functions appear in many places in mathematics.
In the examples that have inspired our project, $K$ is of characteristic 0:

If $W$ is a polynomial module of the general linear group $GL_d=GL_d(K)$, then
its $GL_d$-character is a symmetric polynomial, which in turn gives
$W$ the structure of a graded
vector space. Hence the Hilbert (or Poincar\'e) series of the symmetric algebra $K[W]$
is a nice rational symmetric function.

Nice rational symmetric functions appear as Hilbert series in classical invariant theory. For example,
this holds for the Hilbert series of the pure trace algebra of $n\times n$ generic matrices which is the algebra
of invariants of $GL_n$ acting by simultaneous conjugation on several $n\times n$ matrices.
The mixed trace algebra also has a meaning in classical invariant theory and has a Hilbert series which
is a nice rational symmetric function.

The theorem of Belov \cite{Be} gives that for any
PI-algebra $R$ the Hilbert series of the relatively free algebra
$K\langle Y\rangle/T(R)$, $Y=\{y_1,\ldots,y_d\}$,
where $T(R)$ is the T-ideal of the polynomial identities in $d$ variables of $R$,
is a rational function. Berele \cite{B2} established that the proof of Belov
(as presented in the book by Kanel-Belov and Rowen \cite{KBR}) also implies
that this Hilbert series is a nice rational symmetric function.

Every symmetric function $f(X)$ can be presented as a formal series
\[
f(X)=\sum_{\lambda}m(\lambda)S_{\lambda}(X),
\quad m(\lambda)\in K,
\]
where $S_{\lambda}(X)=S_{\lambda}(x_1,\ldots,x_d)$ is the Schur function
indexed with the partition $\lambda=(\lambda_1,\ldots,\lambda_d)$.

Clearly, it is an interesting combinatorial problem to find the multiplicities
$m(\lambda)$ of an explicitly given symmetric function $f(X)$. This problem is naturally related
to the representation theory of $GL_d$ in characteristic 0 because the Schur functions are
the characters of the irreducible polynomial representations of $GL_d$. Another motivation is that
the multiplicities of the Schur functions
in the Hilbert series of the relatively free
algebra $K\langle Y\rangle/T(R)$, $\text{char}K=0$,
are equal to the multiplicities in the (multilinear) cocharacter sequence
of the polynomial identities of $R$.

Drensky and Genov \cite{DG1} introduced the {\it multiplicity series} $M(f;X)$ of
$f(X)\in K[[X]]^{S_d}$. If
\[
f(X)=\sum_{n_i\geq 0}\alpha(n)X^n=\sum_{\lambda}m(\lambda)S_{\lambda}(X),
\quad m(\lambda)\in K,
\]
then
\[
M(f;X)=\sum_{\lambda}m(\lambda)X^{\lambda}
=\sum_{\lambda_i\geq\lambda_{i+1}}m(\lambda)x_1^{\lambda_1}\cdots x_d^{\lambda_d}
\in K[[X]]
\]
is the generating function of the multiplicities $m(\lambda)$.
Berele, in \cite{B3}, (and also not explicitly stated in \cite{B2}) showed that the multiplicity
series of a nice rational symmetric function $f(X)$ is also a nice rational function.
This fact was one of the key moments in the recent theorem about the exact asymptotics
\[
c_n(R)\simeq an^{k/2}b^n,\quad a\in{\mathbb R},\quad k,b\in {\mathbb N},
\]
of the codimension sequence $c_n(R)$, $n=0,1,2,\ldots$,
of a unital PI-algebra $R$ in characteristic 0
(Berele and Regev \cite{BR} for finitely generated algebras and Berele \cite{B3}
in the general case).
Unfortunately, the proof of Berele does not yield an algorithm to compute the
multiplicity series of $f(X)$. In two variables,  Drensky and Genov \cite{DG2} developed methods
to compute the multiplicity series for nice rational symmetric functions.

The approach of Berele \cite{B2, B3} involves classical results on generating functions
of nonnegative solutions of systems of linear homogeneous equations, obtained by Elliott \cite{E}
and MacMahon \cite{MM}, as stated in the paper by Stanley \cite{S1}.
Going back to the originals \cite{E} and \cite{MM}, we see that
the results there provide algorithms to compute the multiplicity series for nice rational
symmetric functions in any number of variables. The method of Elliott \cite{E} was further developed
by MacMahon \cite{MM} in his ``$\Omega$-Calculus'' or Partition Analysis.
The ``$\Omega$-Calculus'' was improved, with
computer realizations, see Andrews, Paule, and Riese \cite{APR1, APR2}, and Xin \cite{X}. The series of
twelve papers on MacMahon's partition analysis by Andrews, alone or jointly with Paule, Riese, and Strehl (I -- \cite{A}$,\ldots,$
XII -- \cite{AP}) gave a new life of the methods, with numerous applications to different problems.
It seems that for the moment the original approach of \cite{E, MM} and its further developments have not been used
very efficiently in invariant theory and theory of PI-algebras. The only results in this direction we are aware of are in the recent paper
by Bedratyuk and Xin \cite{BX}.

Our computations are based on the ideas of Xin \cite{X} and have been performed
with standard functions of Maple on a usual personal computer.
We illustrate the methods on several problems on symmetric algebras, in classical
invariant theory, algebras with polynomial identities, and noncommutative invariant theory.
The results of Section \ref{section for reduction to MacMahon} hold for any field $K$ of arbitrary characteristic.
In the other sections we assume that $K$ is of characteristic 0.

\section{Reduction to MacMahon's partition analysis}\label{section for reduction to MacMahon}

Recall that one of the ways to define Schur functions (e.g.,  Macdonald \cite{M})
is as fractions of Vandermonde type determinants
\[
S_{\lambda}(X)=\frac{V(\lambda+\delta,X)}{V(\delta,X)},
\]
where $\lambda=(\lambda_1,\ldots,\lambda_d)$,
$\delta=(d-1,d-2,\ldots,2,1,0)$, and
\[
V(\mu,X)=\left\vert\begin{matrix}
x_1^{\mu_1}&x_2^{\mu_1}&\cdots&x_d^{\mu_1}\\
&&&\\
x_1^{\mu_2}&x_2^{\mu_2}&\cdots&x_d^{\mu_2}\\
&&&\\
\vdots&\vdots&\ddots&\vdots\\
&&&\\
x_1^{\mu_d}&x_2^{\mu_d}&\cdots&x_d^{\mu_d}\\
\end{matrix}\right\vert, \quad \mu=(\mu_1,\ldots,\mu_d).
\]
If $f(X)\in K[[X]]^{S_d}$ is a symmetric function, it can be presented
in a unique way as
\[
f(X)=\sum_{\lambda}m(\lambda)S_{\lambda}(X),
\]
where the ``$\lambda$-coordinate'' $m(\lambda)\in K$ is called
the {\it multiplicity} of $S_{\lambda}(X)$.
Our efforts are concentrated around the problem:
{\it Given $f(X)\in K[[X]]^{S_d}$, find the multiplicity series
\[
M(f;X)=\sum_{\lambda}m(\lambda)X^{\lambda}
=\sum_{\lambda_i\geq\lambda_{i+1}}m(\lambda)x_1^{\lambda_1}\cdots x_d^{\lambda_d}
\in K[[X]]
\]
and the multiplicities $m(\lambda)$.}
It is convenient to introduce new variables
\[
v_1=x_1,v_2=x_1x_2,\ldots,v_d=x_1\cdots x_d
\]
and to consider the algebra of formal power series $K[[V]]=K[[v_1,\ldots,v_d]]$
as a subalgebra of $K[[X]]$. As in \cite{DG1}, we introduce the function $M'(f;V)$
(also called the multiplicity series of $f(X)$) by
\[
M'(f;V)=M(f;v_1,v_1^{-1}v_2,\ldots,v_{d-1}^{-1}v_d)
=\sum_{\lambda}m(\lambda)v_1^{\lambda_1-\lambda_2}\cdots v_{d-1}^{\lambda_{d-1}-\lambda_d}v_d^{\lambda_d}.
\]
The mapping $M':K[[X]]^{S_d}\to K[[V]]$ defined by $M':f(X)\to M'(f;V)$ is a bijection.

The proof of the following easy lemma is given in \cite{B2}.

\begin{lemma}\label{lemma of Berele}
Let $f(X)\in K[[X]]^{S_d}$ be a symmetric function and let
\[
g(X)=f(X)\prod_{i<j}(x_i-x_j)=\sum_{r_i\geq 0}\alpha(r_1,\ldots,r_d)x_1^{r_1}\cdots x_d^{r_d},
\quad \alpha(r_1,\ldots,r_d)\in K.
\]
Then the multiplicity series of $f(X)$ is given by
\[
M(f;X)=\frac{1}{x_1^{d-1}x_2^{d-2}\cdots x_{d-2}^2x_{d-1}}
\sum_{r_i>r_{i+1}}\alpha(r_1,\ldots,r_d)x_1^{r_1}\cdots x_d^{r_d},
\]
where the summation is over all $r=(r_1,\ldots,r_d)$ such that
$r_1>r_2>\cdots>r_d$.
\end{lemma}

By the previous lemma, given a nice rational function
\[
g(X)=\sum_{r_i\geq 0}\alpha(r)X^r=p(X)\prod\frac{1}{(1-X^a)^{b_a}}
\]
we start by computing ``half'' of it, i.e., the infinite sum of $\alpha(r)x_1^{r_1}x_2^{r_2}\cdots x_d^{r_d}$
for $r_1>r_2$, and then we continue in the same way with the other variables. To illustrate the method of
Elliott \cite{E}, it is sufficient to consider the case of two variables only.
Given the series
\[
g(x_1,x_2)=\sum_{i,j\geq 0}\alpha_{ij}x_1^ix_2^j
\]
we introduce a new variable $z$ and consider the Laurent series
\[
g(x_1z,\frac{x_2}{z})=\sum_{i,j\geq 0}\alpha_{ij}x_1^ix_2^jz^{i-j}
=\sum_{n=-\infty}^{\infty}g_n(x_1,x_2)z^n,\quad g_n(x_1,x_2)\in K[[x_1,x_2]].
\]
We want to present $g(x_1z,x_2/z)$ as a sum of two series, one in $z$ and the other in $1/z$:
\[
g(x_1z,\frac{x_2}{z})=\sum_{n\geq 0}g_n(x_1,x_2)z^n+\sum_{n>0}g_{-n}(x_1,x_2)\left(\frac{1}{z}\right)^n,
\]
and then take the first summand and replace $z$ with 1 there. If $g(x_1,x_2)$ is a nice rational function,
then $g(x_1,x_2)$ and $g(x_1z,x_2/z)$ have the form
\[
g(x_1,x_2)=p(x_1,x_2)\prod\frac{1}{1-x_1^ax_2^b},\quad p(x_1,x_2)\in K[x_1,x_2],
\]
\[
g(x_1z,\frac{x_2}{z})=p(x_1z,\frac{x_2}{z})\prod\frac{1}{1-x_1^ax_2^bz^{a-b}}.
\]
The expression $\prod 1/(1-x_1^ax_2^bz^{a-b})$ is a product of three factors
\[
\prod_{a_0=b_0}\frac{1}{1-x_1^{a_0}x_2^{b_0}},\quad
\prod_{a_1>b_1}\frac{1}{1-x_1^{a_1}x_2^{b_1}z^{a_1-b_1}},
\quad \prod_{a_2<b_2}\frac{1}{1-x_1^{a_2}x_2^{b_2}/z^{b_2-a_2}}.
\]
If $\prod 1/(1-x_1^ax_2^bz^{a-b})$ contains factors of both the second and the third type,
Elliott \cite{E} suggests to apply the equality
\[
\frac{1}{(1-Az^a)(1-B/z^b)}=\frac{1}{1-ABz^{a-b}}\left(\frac{1}{1-Az^a}+\frac{1}{1-B/z^b}-1\right)
\]
to one of the expressions $1/(1-x_1^{a_1}x_2^{b_1}z^{a_1-b_1})(1-x_1^{a_2}x_2^{b_2}/z^{b_2-a_2})$
and to represent $\prod 1/(1-x_1^ax_2^bz^{a-b})$ as a sum of three expressions which are simpler than the original one.
Continuing in this way, one represents $\prod 1/(1-x_1^ax_2^bz^{a-b})$ as a sum of products of two types:
\[
\prod_{a\geq b}\frac{1}{1-x_1^ax_2^bz^{a-b}}\quad\text{\rm and}\quad
\prod_{a_0=b_0}\frac{1}{1-x_1^{a_0}x_2^{b_0}}\prod_{a_2<b_2}\frac{1}{1-x_1^{a_2}x_2^{b_2}/z^{b_2-a_2}}.
\]
With some additional easy arguments we can represent $g(x_1z,x_2/z)$ as a linear combination of monomials $A_1z^i$,
$i\geq 0$, and quotients of the form
\[
\frac{A_2}{z^j},j>0,\quad B_1z^i\prod_{b\geq 0}\frac{1}{1-B_2z^b},i\geq 0,\quad
\frac{C_1}{z^j}\prod\frac{1}{1-C_2}\prod_{c>0}\frac{1}{1-C_3/z^c}, j\geq 0,
\]
with coefficients $A_1,A_2,B_1,B_2,C_1,C_2,C_3$ which are monomials in $x_1,x_2$.
Comparing this form of $g(x_1z,x_2/z)$ with its expansion as a Laurent series in $z$
\[
g(x_1z,x_2/z)=\sum_{n=-\infty}^{\infty}g_n(x_1,x_2)z^n,
\]
we obtain that the part $\sum_{n\geq 0}g_n(x_1,x_2)z^n$ which we want to compute is the sum of
$A_1z^i$, $B_1z^i\prod_{b\geq 0}1/(1-Bz^b)$ and the fractions $C_1/z^j\prod 1/(1-C_2)$ with $j=0$.

Generalizing the idea of Elliott, in his famous book \cite{MM}
MacMahon defined operators $\mathop{\Omega}\limits_{\geq}$ and $\mathop{\Omega}\limits_{=0}$.
The first operator cuts the negative powers of a Laurent formal power series and then replaces $z$ with 1:
\[
\mathop{\Omega}\limits_{\geq}:\sum_{n_i=-\infty}^{+\infty}\alpha(n)Z^n\to
\sum_{n_i=0}^{+\infty}\alpha(n),
\]
and the second one takes the constant term of series
\[
\mathop{\Omega}\limits_{=0}:\sum_{n_i=-\infty}^{+\infty}\alpha(n)Z^n\to \alpha(0),
\]
where $\alpha(n)=\alpha(n_1,\ldots,n_d)\in K[[X]]$, $\alpha(0)=\alpha(0,\ldots,0)$ and $Z^n=z_1^{n_1}\cdots z_d^{n_d}$.

The next theorem presents the multiplicity series of an arbitrary symmetric function in terms of the Partition
Analysis of MacMahon.

\begin{theorem}\label{Omega version of multiplicity series}
Let $f(X)\in K[[X]]^{S_d}$ be a symmetric function in $d$ variables and let
\[
g(X)=f(X)\prod_{i<j}(x_i-x_j).
\]
Then the multiplicity series of $f(X)$ satisfies
\[
M(f;X)=\frac{1}{x_1^{d-1}x_2^{d-2}\cdots x_{d-2}^2x_{d-1}}
\mathop{\Omega}\limits_{\geq}\left(g(x_1z_1,x_2z_1^{-1}z_2\cdots x_{d-1}z_{d-2}^{-1}z_{d-1},x_dz_{d-1}^{-1})\right).
\]
\end{theorem}

\begin{proof}
Let
\[
g(X)=\sum_{r_i\geq 0}\alpha(r)X^r,
\quad \alpha(r)\in K, X^r=x_1^{r_1}\cdots x_d^{r_d}.
\]
Then
\[
g(x_1z_1,x_2z_1^{-1}z_2\cdots x_{d-1}z_{d-2}^{-1}z_{d-1},x_dz_{d-1}^{-1})
=\sum_{r_i\geq 0}\alpha(r)X^rz_1^{r_1-r_2}z_2^{r_2-r_3}\cdots z_{d-1}^{r_{d-1}-r_d},
\]
\[
\mathop{\Omega}\limits_{\geq}\left(g(x_1z_1,x_2z_1^{-1}z_2\cdots x_{d-1}z_{d-2}^{-1}z_{d-1},x_dz_{d-1}^{-1})\right)
=\sum_{r_i\geq r_{i+1}}\alpha(r)X^r.
\]
The function $g(X)$ is skew-symmetric because $f(X)$ is symmetric. Hence
$\alpha(r)$ is equal to 0, if $r_i=r_j$ for some $i\not= j$ and the summation
in the latter equality for $\mathop{\Omega}\limits_{\geq}$
runs on $r_1>\cdots >r_d$ (and not on $r_1\geq\cdots\geq r_d$). Now the proof follows immediately from
Lemma \ref{lemma of Berele}.
\end{proof}

The $\Omega$-operators were applied by MacMahon \cite{MM} to Elliott rational functions which share many properties with nice rational
functions. He used the Elliott reduction process described above.
The computational approach developed by Andrews, Paule, and Riese \cite{APR1, APR2}
is based on improving this reduction process.
There is another algorithm due to Xin \cite{X} which involves partial fractions.
In this paper we shall use an algorithm inspired by the algorithm of Xin \cite{X}.
We shall state it in the case of two variables. The case of nice rational symmetric functions in several
variables is obtained in an obvious way by multiple application of the algorithm
to the function $g(X)=f(X)\prod_{i<j}(x_i-x_j)$ in $d$ variables instead of to the function
$g(x_1,x_2)=f(x_1,x_2)(x_1-x_2)$ in two variables.

\begin{algorithm}\label{algorithm of Xin}
Let $g(x_1,x_2)\in K[[x_1,x_2]]$ be a nice rational function. In $g(x_1z,x_2/z)$ we replace the factors
$1/(1-C/z^c)$, where $C$ is a monomial in $x_1,x_2$, with the factor $z^c/(z^c-C)$.
Then $g(x_1z,x_2/z)$ becomes a rational function of the form
\[
g(x_1z,\frac{x_2}{z})=\frac{p(z)}{z^a}\prod\frac{1}{1-A}\prod\frac{1}{1-Bz^b}\prod\frac{1}{z^c-C},
\]
where $p(z)$ is a polynomial in $z$ with coefficients which are rational functions in $x_1,x_2$ and $A,B,C$ are monomials in $x_1,x_2$.
Presenting $g(x_1z,x_2/z)$ as a sum of partial fractions with respect to $z$ we obtain that
\[
g(x_1z,\frac{x_2}{z})=p_0(z)+\sum\frac{p_i}{z^i}+\sum\frac{r_{jk}(z)}{q_j(z)^k},
\]
where $p_0(z),r_{jk}(z),q_j(z)\in K(x_1,x_2)[z]$, $p_i\in K(x_1,x_2)$, $q_j(z)$ are the irreducible factors over $K(x_1,x_2)$ of
the binomials $1-Bz^b$ and $z^c-C$ in the expression of $g(x_1z,x_2/z)$, and $\deg_zr_{jk}(z)<\deg_zq_j(z)$. Clearly
$p_0(z)$ gives a contribution to the series $\sum_{n\geq 0}g_n(x_1,x_2)z^n$ in the expansion of $g(x_1z,x_2/z)$ as a
Laurent series. Similarly, $r_{jk}(z)/q_j(z)^k$ contributes to the same series for the factors $q_j(z)$ of
$1-Bz^b$. The fraction $p_i/z^i$ is a part of the series $\sum_{n>0}g_n(x_1,x_2)/z^n$. When $q_j(z)$ is a factor of
$z^c-C$, we obtain that $q_j(z)=z^dq_j'(1/z)$, where $d=\deg_zq_j(z)$ and $q_j'(\zeta)\in K(x_1,x_2)[\zeta]$ is a divisor of
$1-C\zeta^c$. Since $\deg_zr_{jk}(z)<\deg_zq_j(z)$ we derive that $r_{jk}(z)/q_j(z)^k$ contributes to
$\sum_{n>0}g_n(x_1,x_2)/z^n$ and does not give any contribution to $\sum_{n\geq 0}g_n(x_1,x_2)z^n$. Hence
\[
\sum_{n\geq 0}g_n(x_1,x_2)z^n=p_0(z)+\sum\frac{r_{jk}(z)}{q_j(z)^k},
\]
where the sum in the right side of the equation runs on the irreducible divisors $q_j(z)$ of the factors $1-Bz^b$
of the denominator of $g(x_1z,x_2/z)$.
Substituting 1 for $z$ we obtain the expression for $\mathop{\Omega}\limits_{\geq}(g(x_1z,x_2/z))$.
\end{algorithm}

\begin{proof}
The process described in the algorithm gives that
\[
\mathop{\Omega}\limits_{\geq}\left(g(x_1z,\frac{x_2}{z})\right)
=\frac{P(x_1,x_2)}{Q(x_1,x_2)},\quad P(x_1,x_2),Q(x_1,x_2)\in K[x_1,x_2].
\]
By the elimination process of Elliott we already know that $\mathop{\Omega}\limits_{\geq}(g(x_1z,x_2/z))$ is a nice rational function.
Hence the polynomial $Q(x_1,x_2)$ is a divisor of a product of binomials $1-X^a$. Hence the output is in a form
that allows, starting with a nice rational symmetric function $f(X)$ in $d$ variables,
to continue the process with the other variables and to compute the multiplicity series $M(f;X)$.
\end{proof}

\begin{remark}\label{how to check multiplicity series}
If $f(X)$ is a nice rational symmetric function, we can find the multiplicity series
$M(f;X)$ applying Lemma \ref{lemma of Berele} and
using the above algorithm. On the other hand,
it is very easy to check whether the formal power series
\[
h(X)=\sum \beta(q)X^q,\quad
q_1\geq\cdots\geq q_d,
\]
is equal to the multiplicity series $M(f;X)$ of $f(X)$. This is because
$h(X)=M(f;X)$ if and only if
\[
f(X)\prod_{i<j}(x_i-x_j)=\sum_{\sigma\in S_d}\text{\rm sign}(\sigma)
x_{\sigma(1)}^{d-1}x_{\sigma(2)}^{d-2}\cdots x_{\sigma(d-1)}
h(x_{\sigma(1)},\ldots,x_{\sigma(d)}).
\]
These arguments can be used to verify most of our computational results on multiplicities.
\end{remark}

\section{Symmetric algebras}\label{section on symmetric algebras}

Till the end of the paper we assume that $K$ is a field of characteristic 0.
For a background on the representation theory of $GL_d=GL_d(K)$ in the level we need see
the book by Macdonald \cite{M} or the paper by Almkvist, Dicks and Formanek \cite{ADF}.
We fix is a polynomial $GL_d$-module $W$.
Then $W$ is a direct sum
of its irreducible components $W(\mu)$, where $\mu=(\mu_1,\ldots,\mu_d)$ is a partition
in not more than $d$ parts,
\[
W=\bigoplus_{\mu}k(\mu)W(\mu),
\]
where the nonnegative integer $k(\mu)$ is the multiplicity of $W(\mu)$ in the decomposition of $W$.
The vector space $W$ has a basis of eigenvectors of the diagonal subgroup $D_d$ of $GL_d$
an we fix such a basis $\{w_1,\ldots,w_p\}$:
\[
g(w_j)=\xi_1^{\alpha_{j1}}\cdots\xi_d^{\alpha_{jd}}w_j,\quad g=\text{diag}(\xi_1,\ldots,\xi_d)\in D_d,
j=1,\ldots,p,
\]
where $\alpha_{ji}$, $i=1,\ldots,d$, $j=1,\ldots,p$, are nonnegative integers.
The action of $D_d$ on $W$ induces a ${\mathbb Z}^d$-grading on $W$ assuming that
\[
\deg(w_j)=(\alpha_{j1},\ldots,\alpha_{jd}),\quad j=1,\ldots,p.
\]
The polynomial
\[
H(W;X)=\sum_{j=1}^pX^{\alpha_j}=\sum_{j=1}^px_1^{\alpha_{j1}}\cdots x_d^{\alpha_{jd}}
\]
is the Hilbert series of $W$ and has the form
\[
H(W;X)=\sum_{\mu}k(\mu)S_{\mu}(X).
\]
It plays the role of the character of the $GL_d$-module $W$. If the eigenvalues of $g\in GL_d$
are equal to $\zeta_1,\ldots,\zeta_d$, then
\[
\chi_W(g)=\text{tr}_W(g)=H(W;\zeta_1,\ldots,\zeta_d)=\sum_{\mu}k(\mu)S_{\mu}(\zeta_1,\ldots,\zeta_d).
\]
We identify the symmetric algebra $K[W]$ of $W$ with the polynomial algebra in the variables
$w_1,\ldots,w_p$. We extend diagonally the action of $GL_d$ to the symmetric algebra $K[W]$ of $W$ by
\[
g(f(w))=f(g(w)),\quad g\in GL_d, f\in K[W], w\in W.
\]
The ${\mathbb Z}^d$-grading of $W$ induces a ${\mathbb Z}^d$-grading of $K[W]$ and the Hilbert series of $K[W]$
\[
H(K[W];X)=\prod_{j=1}^p\frac{1}{1-X^{\alpha_j}}=\prod_{j=1}^p\frac{1}{1-x_1^{\alpha_{j1}}\cdots x_d^{\alpha_{jd}}}
\]
is a nice rational symmetric function.
Clearly, here we have assumed that $k(0)=0$, (i.e., $\vert\mu\vert=\mu_1+\cdots+\mu_d>0$ in the decomposition of $W$).
Otherwise the homogeneous component of zero degree of $K[W]$ is infinitely dimensional
and the Hilbert series of $K[W]$ is not well defined.
If we present $H(K[W];X)$ as a series of Schur functions
\[
H(K[W];X)=\sum_{\lambda}m(\lambda)S_{\lambda}(X),
\]
then
\[
K[W]=\bigoplus_{\lambda}m(\lambda)W(\lambda).
\]
Hence the multiplicity series of $H(K[W];X)$ carries the information about the decomposition of $K[W]$ as a sum
of irreducible components.

The symmetric function $H(K[W];X)$ is equal to the plethysm
$H(K[Y];X)\circ H(W;X)$
of the Hilbert series of $K[Y_{\infty}]$, $Y_{\infty}=\{y_1,y_2,\ldots\}$, and the Hilbert series of $W$.
More precisely, $H(K[W];X)$ is the part of $d$ variables
of the plethysm
\[
\prod_{i\geq 1}\frac{1}{1-x_i}\circ \sum_{\mu}k(\mu)S_{\mu}(X_{\infty})
=\sum_{n\geq 0}S_{(n)}(X_{\infty})\circ \sum_{\mu}k(\mu)S_{\mu}(X_{\infty})
\]
of the symmetric functions $\sum_{n\geq 0}S_{(n)}$ and $\sum_{\mu}k(\mu)S_{\mu}$
in the infinite set of variables $X_{\infty}=\{x_1,x_2,\ldots\}$.

\begin{example}\label{Multiplicity series of Thrall}
The decomposition of $H(K[W];X)$ as a series of Schur functions is known in few cases only.
The well known identities (see \cite{M}; the third is obtained from the second
applying the Young rule)
\begin{eqnarray*}
H(K[W(2)];X)&=&\prod_{i\leq j}\frac{1}{1-x_ix_j}
\\
&=&\sum_{\lambda} S_{(2\lambda_1,\ldots,2\lambda_d)}(X),
\\
H(K[W(1^2)];X)=\prod_{i<j}\frac{1}{1-x_ix_j}
&=&\sum_{(\lambda_2,\lambda_4,\ldots)} S_{(\lambda_2,\lambda_2,\lambda_4,\lambda_4,\ldots)}(X),
\\
H(K[W(1)\oplus W(1^2)];X)&=&\prod_{i=1}^d\frac{1}{1-x_i}\prod_{i<j}\frac{1}{1-x_ix_j}\\
&=&\sum_{\lambda} S_{(\lambda_1,\ldots,\lambda_d)}(X)
\end{eqnarray*}
give the following expressions for the multiplicity series ($\lfloor r\rfloor$ is the integer part of $r\in\mathbb R$)
\begin{eqnarray*}
M(H(K[W(2)]);X)&=&\sum_{n_i\geq 0}x_1^{2n_1}(x_1x_2)^{2n_2}\cdots (x_1\cdots x_d)^{2n_d}
\\
&=&\prod_{i=1}^d\frac{1}{1-(x_1\cdots x_i)^2},
\\
M(H(K[W(1^2)]);X)&=&\sum_{n_i\geq 0}(x_1x_2)^{n_1}(x_1x_2x_3x_4)^{n_2}\cdots (x_1\cdots x_{2\lfloor d/2\rfloor})^{n_{\lfloor d/2\rfloor}}
\\
&=&\prod_{i=1}^{\lfloor d/2\rfloor}\frac{1}{1-x_1\cdots x_{2i}},
\\
M(H(K[W(1)\oplus W(1^2)]);X)
&=&\sum_{n_i\geq 0}x_1^{n_1}\cdots (x_1\cdots x_d)^{n_d}
\\
&=&\prod_{i=1}^d\frac{1}{1-x_1\cdots x_i}.
\end{eqnarray*}
In the language of the multiplicity series $M'$ we have
\begin{eqnarray*}
M'(H(K[W(2)]);V)&=&\prod_{i=1}^d\frac{1}{1-v_i^2},
\\
M'(H(K[W(1^2)]);V)&=&\prod_{i=1}^{\lfloor d/2\rfloor}\frac{1}{1-v_{2i}},
\\
M'(H(K[W(1)\oplus W(1^2)]);V)
&=&\prod_{i=1}^d\frac{1}{1-v_i}.
\end{eqnarray*}
\end{example}

\begin{example}\label{multiplicities of KW3 from DG1}
The multiplicity series of the Hilbert series of the symmetric algebra
of the $GL_2$-module $W(3)$ was computed (in a quite elaborate way) in \cite{DG1}.
The calculations may be simplified using the methods of \cite{DG2}. Here we illustrate
the advantages of Algorithm \ref{algorithm of Xin}.
Since
\[
S_{(3)}(x_1,x_2)=x_1^3+x_1^2x_2+x_1x_2^2+x_2^3,
\]
the Hilbert series of $K[W(3)]$ is
\[
H(K[W(3)];x_1,x_2)=\frac{1}{(1-x_1^3)(1-x_1^2x_2)(1-x_1x_2^2)(1-x_2^3)}.
\]
We define the function
\[
g(x_1,x_2)=(x_1-x_2)H(K[W(3)];x_1,x_2)
\]
and decompose $g(x_1z,x_2/z)$ as a sum of partial fractions with respect to $z$. The result is
\medskip
\[
\frac{1}{3(1-x_1^2x_2^2)(1-x_1^3x_2^3)(1-x_1z)}
-\frac{1-x_1z-x_1^2x_2^2-2x_1^3x_2^2z}{3(1-x_1^3x_2^3)(1+x_1^2x_2^2+x_1^4x_2^4)(1+x_1z+x_1^2z^2)}
\]
\[
-\frac{x_1^2x_2^2}{(1-x_1^3x_2^3)(1-x_1^6x_2^6)(1-x_1^2x_2z)}
+\frac{x_2}{3(1-x_1^2x_2^2)(1-x_1^3x_2^3)(x_2-z)}
\]
\[
+\frac{x_2(-2z-x_2-x_1^2x_2^2z+x_1^2x_2^3)}{3(1-x_1^3x_2^3)(1+x_1^2x_2^2+x_1^4x_2^4)(x_2^2+x_2z+z^2)}
-\frac{x_1^3x_2^4}{(1-x_1^3x_2^3)(1-x_1^6x_2^6)(x_1x_2^2-z)}.
\]
\medskip

\noindent The factors containing $z$ in the denominators of the first three summands are
$(1-x_1z)$, $(1+x_1z+x_1^2z^2)$, and $(1-x_1^2x_2z)$. Hence these summands
contribute to $\mathop{\Omega}\limits_{\geq}(g(x_1z,x_2/z))$.
We omit the other three summands because the corresponding factors are
$(x_2-z)$, $(x_2^2+x_2z+z^2)$, and $(x_1x_2^2-z)$. Replacing $z$ by 1 we obtain
\begin{eqnarray*}
\mathop{\Omega}\limits_{\geq}(g(x_1z,x_2/z))&=&\frac{1}{3(1-x_1^2x_2^2)(1-x_1^3x_2^3)(1-x_1)}
\\
&&\mbox{}-\frac{1-x_1z-x_1^2x_2^2-2x_1^3x_2^2z}{3(1-x_1^3x_2^3)(1+x_1^2x_2^2+x_1^4x_2^4)(1+x_1+x_1^2)}
\\
&&\mbox{}
-\frac{x_1^2x_2^2}{(1-x_1^3x_2^3)(1-x_1^6x_2^6)(1-x_1^2x_2)}
\\
&=&\frac{x_1(1-x_1^2x_2+x_1^4x_2^2)}{(1-x_1^3)(1-x_1^2x_2)(1-x_1^6x_2^6)}.
\end{eqnarray*}
Hence
\begin{eqnarray*}
M(H(K[W(3)]);x_1,x_2)&=&\frac{1-x_1^2x_2+x_1^4x_2^2}{(1-x_1^3)(1-x_1^2x_2)(1-x_1^6x_2^6)},
\\
M'(H(K[W(3)]);v_1,v_2)&=&\frac{1-v_1v_2+v_1^2v_2^2}{(1-v_1^3)(1-v_1v_2)(1-v_2^6)}.
\end{eqnarray*}
We can rewrite the expression for $M'$ in the form
\begin{eqnarray*}
M'(H(K[W(3)]);v_1,v_2)&=&
\frac{1}{2}\left(\left(\frac{1-v_1v_2+v_1^2v_2^2}{1-v_2^6}+\frac{1+v_1v_2+v_1^2v_2^2}{(1-v_2^3)^2}\right)\frac{1}{1-v_1^3}\right.
\\
&&\mbox{}\left. +\left(\frac{1}{1-v_2^6}-\frac{1}{(1-v_2^3)^2}\right)\frac{1}{1-v_1v_2}\right)
\end{eqnarray*}
and to expand it as a power series in $v_1$ and $v_2$. The coefficient of $v_1^{a_1}v_2^{a_2}$ is equal to the
multiplicity $m(a_1+a_2,a_2)$ of the partition $\lambda=(a_1+a_2,a_2)$.
\end{example}

\begin{example}\label{multiplicities of KWn and others}
In the same way we have computed the multiplicity series
for the Hilbert series of all symmetric algebras $K[W]$ for $\dim(W)\leq 7$ and several cases
for $\dim(W)=8$. For example
\begin{eqnarray*}
M'(H(K[W(4)]))&=&M'\left(\frac{1}{(1-x_1^4)(1-x_1^3x_2)(1-x_1^2x_2^2)(1-x_1x_2^3)(1-x_2^4)}\right)
\\
&=&\frac{1-v_1^2v_2+v_1^4v_2^2}{(1-v_1^4)(1-v_1^2v_2)(1-v_2^4)(1-v_2^6)};
\\
M'(H(K[W(2)\oplus W(2)]))
&=&M'\left(\frac{1}{(1-x_1^2)^2(1-x_1x_2)^2(1-x_2^2)^2}\right)
\\
&=&\frac{1+v_1^2v_2}{(1-v_1^2)^2(1-v_2^2)^3};
\\
M'(H(K[W(3)\oplus W(3)]))
&=&\frac{p(v_1,v_2)}{(1-v_1^3)^2(1-v_1v_2)^2(1-v_2^3)^2(1-v_2^6)^3},
\end{eqnarray*}
where
\begin{multline*}
p(v_1,v_2)=(1-v_2^3+v_2^6)((1+v_1^6v_2^3)(1+v_2^6)-2v_1^3v_2^3(1+v_2^3))
\\
-v_1v_2(1-v_2^3)(2(1-v_2^3-v_2^9)-v_1v_2(4-v_2^3-v_2^9)-v_1^3(1+v_2^6-4v_2^9)+2v_1^4v_2(1+v_2^6-v_2^9)).
\end{multline*}

\noindent
The obtained decompositions can be easily verified using the equation
\[
f(x_1,x_2)=\frac{x_1M'(f;x_1,x_1x_2)-x_2M'(f;x_2,x_1x_2)}{x_1-x_2}
\]
which for $d=2$ is an $M'$-version of the equation in
Remark \ref{how to check multiplicity series}.
\end{example}

Applying our Algorithm \ref{algorithm of Xin} we have obtained the following decompositions of
the Hilbert series of the symmetric algebras of the irreducible $GL_3$-modules $W(\lambda)$, where
$\lambda=(\lambda_1,\lambda_2,\lambda_3)$ is a partition of 3.
Again, for the proof one can use Remark \ref{how to check multiplicity series}.

\begin{theorem}\label{symmetric algebras in three variables}
Let $d=3$ and let $v_1=x_1$, $v_2=x_1x_2$, $v_3=x_1x_2x_3$. The Hilbert series
\begin{eqnarray*}
H(K[W(3)];x_1,x_2,x_3)&=&\prod_{i\leq j\leq k}\frac{1}{1-x_ix_jx_k},
\\
H(K[W(2,1)];x_1,x_2,x_3)&=&\frac{1}{(1-x_1x_2x_3)^2}\prod_{i\not= j}\frac{1}{1-x_i^2x_j},
\\
H(K[W(1,1,1)];x_1,x_2,x_3)&=&\frac{1}{1-x_1x_2x_3}
\end{eqnarray*}
have the following multiplicity series
\begin{eqnarray*}
M'(H(K[W(3)]);v_1,v_2,v_3)&=& \frac{1}{q(v_1,v_2,v_3)}\sum_{i=0}^8v_1^ip_i(v_2,v_3),
\\
M'(H(K[W(2,1)]);v_1,v_2,v_3)&=& \frac{1+v_1v_2v_3+(v_1v_2v_3)^2}{(1-v_1v_2)(1-v_1^3v_3^2)(1-v_2^3v_3)(1-v_3^2)(1-v_3^3)},
\\
M'(H(K[W(1,1,1)]);v_1,v_2,v_3)&=& \frac{1}{1-v_3},
\end{eqnarray*}
where
\begin{eqnarray*}
q&=&(1-v_1^3)(1-v_1v_2)(1-v_1^3v_3^2)(1-v_1^3v_3^3)(1-v_2^6)(1-v_2^3v_3)(1-v_2^3v_3^3)
\\
&&\mbox{}\times (1-v_3^4)(1-v_3^6),
\\
p_0 &=& 1+v_2^9v_3^6,
\\
p_1&=&-v_2(1-v_2^3v_3^2(1+v_3^2+v_3^3)-v_2^6v_3^2(1+v_3^2)+v_2^9v_3^5(1+v_3)),
\\
p_2&=&v_2^2((1+v_3^2+v_3^4)-v_2^3v_3^5(1-v_3)-v_2^6v_3^2(1+v_3^2)+v_2^9v_3^5),
\\
p_3&=&-((1-v_2^3v_3^3)v_3+v_2^6+v_2^9v_3^4(1+v_3^2+v_3^3))v_3^2,
\\
p_4&=&v_2v_3^2((1+2v_3+v_3^3)-v_2^3v_3(v_3+v_2^3)(1+v_3^2)(1+v_3+v_3^2)
\\
&&\mbox{}+v_2^9v_3^4(1+2v_3^2+v_3^3)),
\\
p_5&=&-v_2^2v_3^2(1+v_3+v_3^3+v_2^3v_3^7-v_2^6v_3^3+v_2^9v_3^6),
\\
p_6&=&(v_3-v_2^3v_3^2(1+v_3^2)+v_2^6(1-v_3)+v_2^9v_3^2(1+v_3^2+v_3^4))v_3^5,
\\
p_7&=&-v_2v_3^5(1+v_3-v_2^3v_3^2(1+v_3^2)-v_2^6v_3(1+v_3+v_3^3)+v_2^9v_3^6),
\\
p_8&=&v_2^2v_3^5(1+v_2^9v_3^6).
\end{eqnarray*}
\end{theorem}

For a polynomial $GL_d$-module $W=\sum k(\mu)W(\mu)$, the symmetric algebra $K[W]$ has also another natural $\mathbb Z$-grading
induced by the assumption that the elements of $W$ are of first degree. Then the homogeneous component
$K[W]^{(n)}$ of degree $n$ is the symmetric tensor power $W^{\otimes_sn}$ and
\[
K[W]=\bigoplus_{n\geq 0}\bigoplus_{\lambda}m_n(\lambda)W(\lambda),\quad m_n(\lambda)\in\mathbb Z.
\]
In order to take into account both the ${\mathbb Z}^d$-grading induced by the $GL_d$-action and the
natural $\mathbb Z$-grading of $K[W]$, we introduce an additional variable in the Hilbert series of $K[W]$:
\[
H(K[W];X,t)=\sum_{n\geq 0}H(K[W]^{(n)};X)t^n.
\]
If, as above, the Hilbert series of $W$ is
\[
H(W;X)=\sum_{\mu}k(\mu)S_{\mu}(X)=\sum_{j=1}^pX^{\alpha_j},
\]
then
\[
H(K[W];X,t)=\prod_{j=1}^p\frac{1}{1-X^{\alpha_j}t}
=\sum_{n\geq 0}\left(\sum_{\lambda}m_n(\lambda)S_{\lambda}(X)\right)t^n.
\]
Hence the multiplicity series
\[
M(H(K[W]);X,t)=\sum_{n\geq 0}\left(\sum_{\lambda}m_n(\lambda)X^{\lambda}\right)t^n
\]
carries the information about the multiplicities of the irreducible $GL_d$-submodules in the homogeneous
components $K[W]^{(n)}$ of $K[W]$. A minor difference with the nongraded case is that
we allow $\vert\mu\vert=0$ in the decomposition of $W$: Since $W$ is finite dimensional, the homogeneous
components of $K[W]$ are also finite dimensional and the Hilbert series $H(K[W];X,t)$ is well defined
even if $\vert\mu\vert=0$ for some of the summands $W(\mu)$ of $W$.
In the next section we shall see the role of this multiplicity series
in invariant theory.

\begin{example}\label{SWgraded}
Let $d=2$ and let $W=W(3)\oplus W(2)$. Then the Hilbert series of $W$ is
\begin{eqnarray*}
H(W;x_1,x_2)&=&S_{(3)}(x_1,x_2)+S_{(2)}(x_1,x_2)
\\
&=&x_1^3+x_1^2x_2+x_1x_2^2+x_2^3+x_1^2+x_1x_2+x_2^2
\end{eqnarray*}
and the Hilbert series of $K[W]$ which takes into account also the $\mathbb Z$-grading is
\[
H(K[W];x_1,x_2,t)=\prod_{a+b=3}\frac{1}{1-x_1^ax_2^bt}\prod_{a+b=2}\frac{1}{1-x_1^ax_2^bt}.
\]
Applying Algorithm \ref{algorithm of Xin} we obtain that $M'(H(K[W]);x_1,x_2,t)$ is equal to
\[
\frac{\sum_{k=0}^{10}p_k(v_1,v_2)t^k}
{(1-v_1^2t)(1-v_1^3t)(1-v_1v_2t)(1-v_2^2t^2)(1-v_2^4t^3)(1-v_2^6t^4)(1-v_2^6t^5)},
\]
where
\[
p_0=1,\quad p_1=-v_1v_2,\quad p_2=v_1v_2(v_2+v_1v_2+v_1^2), \quad p_3=v_1v_2^2(v_2-v_1^3),
\]
\[
p_4=v_1v_2^5,\quad p_5=v_1v_2^5(1-v_1)(v_2+v_1),\quad p_6=-v_1^3v_2^6,
\]
\[
p_7=v_2^8(v_2-v_1^3),\quad p_8=-v_1v_2^9(v_2+v_1+v_1^2),\quad p_9=v_1^3v_2^{10},\quad p_{10}=-v_1^4v_2^{11}.
\]
Again, we can check directly, without Algorithm \ref{algorithm of Xin}
that the obtained rational function is
the multiplicity series of $H(K[W];x_1,x_2,t)$ using
Remark \ref{how to check multiplicity series}.
\end{example}

We complete this section with the $\mathbb Z$-graded version for the multiplicities
of the symmetric algebra of $W=W(1)\oplus W(1^2)$ for arbitrary positive integer $d$.

\begin{proposition}\label{analogue of Thrall}
Let $d$ be any positive integer. The homogeneous component of degree $n$ of the symmetric algebra
$K[W(1)\oplus W(1^2)]$ of the $GL_d$-module $W=W(1)\oplus W(1^2)$ decomposes as
\[
K[W(1)\oplus W(1^2)]^{(n)}=\bigoplus W(\lambda),
\]
where the summation runs on all $\lambda=(\lambda_1,\ldots,\lambda_d)$ such that
$\lambda_1+\lambda_3+\cdots+\lambda_{2\lfloor (d-1)/2\rfloor+1}=n$.
Equivalently, the multiplicity series of $H(K[W(1)\oplus W(1^2)]);X,t)$ is
\[
M'(H(K[W(1)\oplus W(1^2)]);V,t)=\prod_{2i\leq d}\frac{1}{(1-v_{2i-1}t^i)(1-v_{2i}t^i)},
\]
when $d$ is even and
\[
M'(H(K[W(1)\oplus W(1^2)]);V,t)=\frac{1}{1-v_dt^{(d+1)/2}}\prod_{2i<d}\frac{1}{(1-v_{2i-1}t^i)(1-v_{2i}t^i)}
\]
when  $d$ is odd.
\end{proposition}

\begin{proof}
Since, as $GL_d$-module,
\[
K[W(1)\oplus W(1^2)]=K[W(1)]\otimes K[W(1^2)],
\]
we obtain that
\begin{eqnarray*}
H(K[W(1)\oplus W(1^2)];X,t)&=&H(K[W(1)];X,t)H(K[W(1^2)];X,t)
\\
&=&\prod_{i=1}^d\frac{1}{1-x_it}\prod_{i<j}\frac{1}{1-x_ix_jt}
\end{eqnarray*}
because the elements of $W(1),W(1^2)\subset W$ are of first degree.
The decompositions
\begin{eqnarray*}
\prod_{i=1}^d\frac{1}{1-x_i}&=&\sum_{k\geq 0}S_{(k)}(X),
\\
\prod_{i<j}\frac{1}{1-x_ix_j}
&=&\sum_{(\lambda_2,\lambda_4,\ldots)} S_{(\lambda_2,\lambda_2,\lambda_4,\lambda_4,\ldots)}(X)
\end{eqnarray*}
give
\begin{eqnarray*}
\prod_{i=1}^d\frac{1}{1-x_it}&=&\sum_{m\geq 0}S_{(m)}(X)t^m,
\\
\prod_{i<j}\frac{1}{1-x_ix_jt}
&=&\sum_{(\lambda_2,\lambda_4,\ldots)} S_{(\lambda_2,\lambda_2,\lambda_4,\lambda_4,\ldots)}(X)t^{\lambda_2+\lambda_4+\cdots}
\end{eqnarray*}
and hence
\[
H(K[W(1)\oplus W(1^2)];X,t)
=\sum_{m\geq 0}\sum_{(\lambda_2,\lambda_4,\ldots)}S_{(m)}(X)
S_{(\lambda_2,\lambda_2,\lambda_4,\lambda_4,\ldots)}(X)t^{m+\lambda_2+\lambda_4+\cdots}.
\]
The product of the Schur functions $S_{(m)}(X)$ and $S_{\mu}(X)$ can be decomposed by the Young rule
which is a partial case of the Littlewood--Richardson rule:
\[
S_{(m)}(X)S_{\mu}(X)=\sum S_{\nu}(X),
\]
where the summation runs on all partitions $\nu\vdash m+\vert\mu\vert$ such that
\[
\nu_1\geq\mu_1\geq \nu_2\geq\mu_2\geq\cdots\geq\nu_d\geq\mu_d.
\]
Applied to our case this gives
\[
S_{(m)}(X)S_{(\lambda_2,\lambda_2,\lambda_4,\lambda_4,\ldots)}(X)=\sum S_{(\lambda_1,\lambda_2,\ldots,\lambda_d)}(X),
\]
where the sum is on all partitions with
\[
(\lambda_1-\lambda_2)+(\lambda_3-\lambda_4)+\cdots+(-1)^{d-1}\lambda_d =m.
\]
Hence the $GL_d$-module $W(\lambda)\subset K[W]$ is a submodule of the homogeneous component of degree
\[
m+\lambda_2+\lambda_4+\cdots=\lambda_1+\lambda_3+\cdots+\lambda_{2\lfloor (d-1)/2\rfloor+1}.
\]
For the statement for the multiplicity series, first let $d=2k+1$. Then
\begin{eqnarray*}
M(H(K[W]);X,t)&=&\sum_{\lambda_1\geq\cdots\geq\lambda_{2k+1}}x_1^{\lambda_1}\cdots x_{2d+1}^{\lambda_{2k+1}}t^{\lambda_1+\lambda_3+\cdots+\lambda_{2k+1}}
\\
&=&\sum_{\lambda_1\geq\cdots\geq\lambda_{2k+1}}(x_1t)^{\lambda_1}x_2^{\lambda_2}(x_3t)^{\lambda_3}\cdots x_{2d}^{\lambda_{2k+1}}(x_{2d+1}t)^{\lambda_{2k+1}}
\\
&=&\sum_{\lambda_1\geq\cdots\geq\lambda_{2k+1}}(x_1t)^{\lambda_1-\lambda_2}(x_1x_2t)^{\lambda_2-\lambda_3}(x_1x_2x_3t^2)^{\lambda_3-\lambda_4}
\\
&&\mbox{}\times
(x_1x_2x_3x_4t^2)^{\lambda_4-\lambda_5}\cdots
 (x_1\cdots x_{2k-1}t^k)^{\lambda_{2k-1}-\lambda_{2k}}
 \\ &&\mbox{}\times (x_1\cdots x_{2k}t^k)^{\lambda_{2k}-\lambda_{2k+1}}
(x_1\cdots x_{2k+1}t^{k+1})^{\lambda_{2k+1}},
\\
M'(H(K[W]);X,t)
&=&\sum_{n_i\geq 0}(v_1t)^{n_1}(v_2t)^{n_2}(v_3t^2)^{n_3}(v_4t^2)^{n_4}\cdots
\\
&&\mbox{}\times (v_{2k-1}t^k)^{n_{2k-1}}(v_{2k}t^k)^{n_{2k}}(v_{2k+1}t^{k+1})^{n_{2k+1}}
\\
&=&\frac{1}{1-v_{2k+1}t^{k+1}}\prod_{i=1}^k\frac{1}{(1-v_{2i-1}t^i)(1-v_{2i}t^i)}.
\end{eqnarray*}
The case $d=2k$ follows immediately from the case $d=2k+1$ by substituting $v_{2k+1}=0$ in the expression of $M'(H(K[W]);X,t)$.
\end{proof}

\section{Invariant theory}\label{section on classical invariant theory}

Without being comprehensive, we shall survey few results related with our topic.
One of the main objects in invariant theory in the 19th century is the algebra
of $SL_2$-invariants of binary forms. Let $W_m=W_{m,2}$ be the vector space of all homogeneous polynomials
of degree $m$ in two variables with the natural action of $SL_2$. The computation of
the Hilbert series
(often also called the Poincar\'e series)
of the algebra of invariants
$K[W_m]^{SL_2}$ was a favorite problem actively also studied nowadays.
It was computed by Sylvester and Franklin \cite{SF1, SF2} for $m\leq 10$ and $m=12$.
In 1980 Springer \cite{Sp} found an explicit formula for the Hilbert series of $K[W_m]^{SL_2}$.
Applying it, Brouwer and Cohen \cite{BC} calculated the Hilbert series of $K[W_m]^{SL_2}$
of degree $\leq 17$. Littelmann and Procesi \cite{LP} suggested an algorithm based on a
variation of the result of Springer and computed the Hilbert series for $m=4k\leq 36$.
More recently, Djokovi\'c \cite{Dj1} proposed a heuristic algorithm for fast computation of the Hilbert series
of the invariants of binary forms, viewed as rational functions, and
computed the series for $m\leq 30$.

Not too much is known about the explicit form of the invariants and their Hilbert series
when $SL_d({\mathbb C})$, $d\geq 3$, acts on the vector space
of forms of degree $m\geq 3$. Most of the known results are for ternary forms.
The generators of the algebra of invariants in the
case of forms of degree 3 were found by Gordan \cite{G}, see also Clebsch and Gordan \cite{CG}; the case of forms of degree 4
has handled by Emmy Noether \cite{No}.
The Hilbert series of the algebra of $SL_3({\mathbb C})$-invariants for forms of degree 4 was calculated by Shioda \cite{Sh}.
Recently Bedratyuk \cite{Bd1, Bd2} found analogues of Sylvester--Cayley and Springer formulas for invariants also of ternary forms.
This allowed him to compute the first coefficients (of the terms of degree $\leq 30$)
of the Hilbert series of the algebras of $SL_3({\mathbb C})$-invariants
of forms of degree $m\leq 7$.

Computing the Hilbert series of the algebra of $SL_d$-invariants ${\mathbb C}[W]^{SL_d}$,
where $W$ is a direct sum of several vector spaces $W_{m_i,d}$ of forms of degree $m_i$ in $d$ variables,
one may use the Molien--Weyl integral formula, evaluating multiple integrals.
This type of integrals can be evaluated using the Residue Theorem, see the book by Derksen and Kemper \cite{DeK} for details.
For concrete decompositions of $W$, the algebra of invariants
$K[W]^{SL_d}$  was studied already by Sylvester.
Its Hilbert series is also known in some cases. For example, recently
Bedratyuk \cite{Bd4} has found a formula for the Hilbert series of the $SL_2$-invariants
$K[W_{m_1,2}\oplus W_{m_2,2}]^{SL_2}$ and has computed these series for $m_1,m_2\leq 20$.
(The results for $m_1,m_2\leq 5$ are given explicitly in \cite{Bd4}.)
Very recently, Bedratyuk and Xin \cite{BX} applied the MacMahon partition analysis to the Molien--Weyl integral formula
and computed the Hilbert series of the algebras of invariants
of some ternary and quaternary forms.

Our approach to the Hilbert series of the algebra of invariants $K[W]^{SL_d}$ of the $SL_d$-module $W$
is based on a theorem of De Concini, Eisenbud and Procesi \cite{DEP}. This theorem implies that
the multiplicities of $S_{\lambda}(X)$ in the Hilbert series of symmetric algebras $K[W]$
of a $GL_d$-module $W$ appear in invariant theory of $SL_d$ and of
the unitriangular group $UT_d=UT_d(K)$ as subgroups of $GL_d$, and in invariant theory of
a single unitriangular matrix. It is combined with
an idea used by Drensky and Genov \cite{DG2} to compute the Hilbert series of the algebra of invariants of
$UT_2$. We extend the $SL_d$-action on $W$ to a polynomial action of $GL_d$.
This is possible in the cases that we consider because the $SL_d$-module $W_{m,d}$ of the forms of degree $m$
can be viewed as a $GL_d$-module isomorphic to $W(m)$. Then we
compute the Hilbert series of the $GL_d$-module $K[W]$ and its multiplicity series $M'(H(K[W]);V,t)$. The
Hilbert series of $K[W]^{SL_d}$ is equal to $M'(H(K[W]);0,\ldots,0,1,t)$. Similarly, if $W$ is a polynomial
$GL_d$-module, then the Hilbert series of $K[W]^{UT_d}$ is equal to $M'(H(K[W]);1,\ldots,1,t)$. The difference with \cite{DG2}
is that there we use for the evaluation of $M(H(K[W]);x_1,x_2,t)$ the
methods developed in \cite{DG2} and here we use the MacMahon partition analysis for the same purpose and for any number of variables.
We shall consider the following problem. {\it Let $W$ be an arbitrary polynomial $GL_d$-module.
How can one calculate the Hilbert series of the algebras of invariants $K[W]^{SL_d}$ and $K[W]^{UT_d}$?}
Clearly, here we assume that $SL_d$ and $UT_d$ are canonically embedded into $GL_d$.
We need the following easy argument. We state it as a lemma and omit the obvious proof.

\begin{lemma}\label{reduction of invariants to irreducible components}
Let $H$ be a subgroup of the group $G$ and let $W_1,W_2$ be $G$-modules. Then the vector space of invariants $W^H\subset W$
in $W=W_1\oplus W_2$ satisfy
\[
W^H=W_1^H\oplus W_2^H.
\]
\end{lemma}

\begin{theorem}\label{SL- and UT-invariants}
Let $W$ be a polynomial $GL_d$-module with Hilbert series with respect to the grading
induced by the $GL_d$-action on $W$
\[
H(W;X)=\sum a_ix_1^{i_1}\cdots x_d^{i_d},
\quad a_i\geq 0,a_i\in\mathbb Z,
\]
and let
\[
H(K[W];X,t)=\prod\frac{1}{(1-X^it)^{a_i}}
\]
be the Hilbert series of $K[W]$ which counts also the natural $\mathbb Z$-grading.
Then the Hilbert series of the algebras of invariants $K[W]^{SL_d}$ and $K[W]^{UT_d}$
are given by
\begin{eqnarray*}
H(K[W]^{SL_d},t)&=&M'(H(K[W]);0,\ldots,0,1,t),
\\
H(K[W]^{UT_d},t)&=&M(H(K[W]);1,\ldots,1,t)
\\
&=&M'(H(K[W]);1,\ldots,1,t).
\end{eqnarray*}
\end{theorem}

\begin{proof}
Let
\[
K[W]=\bigoplus_{n\geq 0}\bigoplus_{\lambda}m_n(\lambda)W(\lambda)
\]
be the decomposition of the $\mathbb Z$-graded $GL_d$-module $K[W]$. Its Hilbert series is
\[
H(K[W];X,t)=\sum_{n\geq 0}\left(\sum_{\lambda}m_n(\lambda)S_{\lambda}(X)\right)t^n
\]
and the multiplicity series of $H(K[W];X,t)$ are
\begin{eqnarray*}
M(H(K[W]);X,t)&=&\sum_{n\geq 0}\left(\sum_{\lambda}m_n(\lambda)X^{\lambda}\right)t^n,
\\
M'(H(K[W]);V,t)&=&\sum_{n\geq 0}
\left(\sum_{\lambda}m_n(\lambda)v_1^{\lambda_1-\lambda_2}\cdots v_{d-1}^{\lambda_{d-1}-\lambda_d}v_d^{\lambda_d}\right)t^n.
\end{eqnarray*}
It is a well known fact that the irreducible $GL_d$-module $W(\lambda)=W(\lambda_1,\ldots,\lambda_d)$
contains a one-dimensional $SL_d$-invariant subspace if $\lambda_1=\cdots=\lambda_d$
(when $\dim(W(\lambda)=1$ and $W(\lambda)^{SL_d}=W(\lambda)$) and contains no invariants
if $\lambda_j\not=\lambda_{j+1}$ for some $j$. Applying Lemma \ref{reduction of invariants to irreducible components}
we immediately obtain
\begin{eqnarray*}
K[W]^{SL_d}&=&\bigoplus_{n\geq 0}\bigoplus_{\lambda_1=\cdots=\lambda_d}m_n(\lambda)W(\lambda),
\\
H(K[W]^{SL_d};t)&=&\sum_{n\geq 0}\left(\sum_{\lambda_1=\cdots=\lambda_d}m_n(\lambda)\right)t^n.
\end{eqnarray*}
Evaluating the monomials in the expansion of $M'(H(K[W]);V,t)$ for $v_1=\cdots=v_{d-1}=0$, $v_d=1$ we obtain
\[
v_1^{\lambda_1-\lambda_2}\cdots v_{d-1}^{\lambda_{d-1}-\lambda_d}v_d^{\lambda_d}\vert_{V=(0,\ldots,0,1)}=
\begin{cases}
1,\text{ if }\lambda_1=\cdots=\lambda_d,\\
0,\text{ if }\lambda_j\not=\lambda_{j+1}\text{ for some }j\\
\end{cases}
\]
which completes the case of $SL_d$-invariants.

It is also well known that every irreducible $GL_d$-module $W(\lambda)$ has a one-dimensional $UT_d$-invariant subspace
which is spanned on the only (up to a multiplicative constant) element $w\in W(\lambda)$ with the property that
the diagonal subgroup $D_d$ of $GL_d$ acts by
\[
g(w)=\xi_1^{\lambda_1}\cdots\xi_d^{\lambda_d}w,\quad g=\text{diag}(\xi_1,\ldots,\xi_d).
\]
Hence
\begin{eqnarray*}
K[W]^{UT_d}&=&\bigoplus_{n\geq 0}\bigoplus_{\lambda}m_n(\lambda)W(\lambda)^{UT_d},
\\
H(K[W]^{UT_d};t)&=&\sum_{n\geq 0}\left(\sum_{\lambda}m_n(\lambda)\right)t^n
\\
&=&M(H(K[W]);1,\ldots,1,t)=M'(H(K[W]);1,\ldots,1,t).
\end{eqnarray*}
\end{proof}

Below we shall illustrate Theorem \ref{SL- and UT-invariants} on the Hilbert series of the $SL_2$-invariants for the $GL_2$-modules
considered in the examples of Section \ref{section on symmetric algebras}.

\begin{example}\label{invariants for homogeneous modules}
If the polynomial $GL_d$-module $W$ is homogeneous of degree $m$, i.e.,
$g(w)=\xi^mw$ for $w\in W$ and $g=\text{diag}(\xi,\ldots,\xi)\in GL_d$, then
\begin{eqnarray*}
H(K[W];x_1,\ldots,x_d,t)&=&H(K[W];X,t)
\\
&=&H(K[W];X\sqrt[m]{t})=H(K[W];x_1\sqrt[m]{t},\ldots,x_d\sqrt[m]{t})
\end{eqnarray*}
because the elements of $W$ are of degree 1 with respect to the $\mathbb Z$-grading and of degree $m$
with respect to the ${\mathbb Z}^d$-grading.
The results of Example \ref{Multiplicity series of Thrall} give
\begin{eqnarray*}
M'(H(K[W(2)]);V)&=&\prod_{i=1}^d\frac{1}{1-v_i^2},
\\
M'(H(K[W(2)]);V,t)&=&\prod_{i=1}^d\frac{1}{1-v_i^2t^i},
\\
H(K[W(2)]^{SL_d};t)&=&M'(H(K[W(2)]);0,\ldots,0,1,t)=\frac{1}{1-t^d},
\\
H(K[W(2)]^{UT_d};t)&=&M'(H(K[W(2)]);1,\ldots,1,t)=\prod_{i=1}^d\frac{1}{1-t^i};
\\
M'(H(K[W(1^2)]);V)&=&\prod_{i=1}^{\lfloor d/2\rfloor}\frac{1}{1-v_{2i}},
\\
M'(H(K[W(1^2)]);V,t)&=&\prod_{i=1}^{\lfloor d/2\rfloor}\frac{1}{1-v_{2i}t^i},
\\
H(K[W(1^2)]^{SL_d};t)&=&\begin{cases}
\displaystyle{\frac{1}{1-t^{d/2}}},\text{ if }d\text{ is even},\\
\\
0,\text{ if }d\text{ is odd},\\
\end{cases}
\\
H(K[W(1^2)]^{UT_d};t)&=&\prod_{i=1}^{\lfloor d/2\rfloor}\frac{1}{1-t^i}.
\end{eqnarray*}
Similarly, for $d=2$ Examples \ref{multiplicities of KW3 from DG1} and \ref{multiplicities of KWn and others} give
\begin{eqnarray*}
M'(H(K[W(3)]);v_1,v_2)&=&\frac{1-v_1v_2+v_1^2v_2^2}{(1-v_1^3)(1-v_1v_2)(1-v_2^6)},
\\
M'(H(K[W(3)]);v_1,v_2,t)&=&M'(H(K[W(3)];v_1\sqrt[3]{t},v_2\sqrt[3]{t^2})
\\
&=&\frac{1-v_1v_2t+v_1^2v_2^2t^2}{(1-v_1^3t)(1-v_1v_2t)(1-v_2^6t^4)},
\\
H(K[W(3)]^{SL_2};t)&=&M'(H(K[W(3)];0,1,t)=\frac{1}{1-t^4},
\\
H(K[W(3)]^{UT_2};t)&=&M'(H(K[W(3)];1,1,t)=\frac{1-t+t^2}{(1-t)^2(1-t^4)};
\\
M'(H(K[W(4)]);v_1,v_2,t)&=&M'(H(K[W(4)]);v_1\sqrt[4]{t},v_2\sqrt{t}),
\\
H(K[W(4)]^{SL_2};t)&=&M'(H(K[W(4)]);0,1,t)=\frac{1}{(1-t^2)(1-t^4)},
\\
H(K[W(4)]^{UT_2};t)&=&M'(H(K[W(4)]);1,1,t)
\\
&=&\frac{1-t+t^2}{(1-t)^2(1-t^2)(1-t^4)};
\\
M'(H(K[W(2)\oplus W(2)]);v_1,v_2,t)&=&M'(H(K[W(2)\oplus W(2)]);v_1\sqrt{t},v_2t),
\\
H(K[W(2)\oplus W(2)]^{SL_2};t)&=&\frac{1}{(1-t^2)^3},
\\
H(K[W(2)\oplus W(2)]^{UT_2};t)&=&\frac{1+t^2}{(1-t)^2(1-t^2)^3};
\\
M'(H(K[W(3)\oplus W(3)]):,v_1,v_2,t)&=&M'(H(K[W(3)\oplus W(3)]):,v_1\sqrt[3]{t},v_2\sqrt[3]{t^2}),
\\
H(K[W(3)\oplus W(3)]^{SL_2};t)&=&\frac{(1-t^2+t^4)(1+t^4)}{(1-t^2)^5(1+t^2)^3},
\\
H(K[W(3)\oplus W(3)]^{UT_2};t)
&=&\frac{1+t^{10}+3t^2(1+t^6)+6t^3(1+t+t^2+t^3+t^4)}{(1-t)^2(1-t^2)^5(1+t^2)^3}.
\end{eqnarray*}
Finally, for $d=3$ Theorem \ref{symmetric algebras in three variables} gives that
\begin{eqnarray*}
M'(H(K[W(3)]);v_1,v_2,v_3,t)&=&M'(H(K[W(3)]);v_1\sqrt[3]{t},v_2\sqrt[3]{t^2},v_3t),
\\
(H(K[W(3)]^{SL_3};t)&=&\frac{1}{(1-t^4)(1-t^6)},
\\
(H(K[W(3)]^{UT_3};t)&=&\frac{(1+t^3)(1+t^9)+2t^4(1+t^4)+3t^5(1+t+t^2)}{(1-t)(1-t^2)(1-t^3)^2(1-t^4)^2(1-t^5)}
\end{eqnarray*}
and similarly
\begin{eqnarray*}
H(K[W(2,1)]^{SL_3};t)&=&\frac{1}{(1-t^2)(1-t^3)},
\\
H(K[W(2,1)]^{UT_3};t)&=&\frac{1-t+t^2}{(1-t)^2(1-t^2)(1-t^3)^2};
\\
H(K[W(1^3)]^{SL_3};t)&=&H(K[W(1^3)]^{UT_3};t)=\frac{1}{1-t}.
\end{eqnarray*}
\end{example}

\begin{example}\label{invariants - nonhomogeneous case}
The translation of Example \ref{SWgraded} to the language of $SL_2$- and $UT_2$-invariants gives
\begin{eqnarray*}
H(K[W(3)\oplus W(2)]^{SL_2};t)&=&\frac{1+t^9}{(1-t^2)(1-t^3)(1-t^4)(1-t^5)},
\\
H(K[W(3)\oplus W(2)]^{UT_2};t)&=&\frac{(1-t)(1-t^7)+4t^2(1+t^4)-t^3(1+t^2)+5t^4}{(1-t)^3(1-t^3)(1-t^4)(1-t^5)}.
\end{eqnarray*}
For an arbitrary $d$ Proposition \ref{analogue of Thrall} gives
\begin{eqnarray*}
H(K[W(1)\oplus W(1^2)]^{SL_d};t)&=&\begin{cases}
\displaystyle{\frac{1}{1-t^k}},\text{ if }d=2k,\\
\\
\displaystyle{\frac{1}{1-t^{k+1}}},\text{ if }d=2k+1,\\
\end{cases}
\\
H(K[W(1)\oplus W(1^2)]^{UT_d};t)&=&\begin{cases}
\displaystyle{\prod_{i=1}^k\frac{1}{(1-t^i)^2}},\text{ if }d=2k,\\
\\
\displaystyle{\frac{1}{1-t^{k+1}}\prod_{i=1}^k\frac{1}{(1-t^i)^2}},\text{ if }d=2k+1.\\
\end{cases}
\end{eqnarray*}
\end{example}

In the above examples our results coincide with the known ones, see e.g., \cite{SF1, Di1, Dj1, Bd4, Bd1}.

The invariant theory of $UT_2$ may be restated in the language of
linear locally nilpotent derivations. Recall that a derivation of a (not necessarily commutative or associative) algebra $R$
is a linear operator $\delta$ with the property that
\[
\delta(u_1u_2)=\delta(u_1)u_2+u_1\delta(u_2),\quad u_1,u_2\in R.
\]
The derivation $\delta$ is locally nilpotent if for any $u\in R$ there exists a $p$ such that $\delta^p(u)=0$.
Locally nilpotent derivations are interesting objects with relations to invariant theory, the 14th Hilbert problem,
automorphisms of polynomial algebras, the Jacobian conjecture, etc., see the monographs by Nowicki \cite{N}, van den Essen \cite{Es},
and Freudenburg \cite{Fr}. Linear locally nilpotent derivations $\delta$ of the polynomial algebra $K[Y]$
(acting as linear operators on the vector space $KY$ with basis $Y=\{y_1,\ldots,y_d\}$) were studied by Weitzenb\"ock \cite{W}
who proved that the algebra of constants $K[Y]^{\delta}$, i.e., the kernel of $\delta$, is finitely generated.
Nowadays linear locally nilpotent derivations of $K[Y]$ are known as {\it Weitzenb\"ock derivations}
and are subjects of intensive study.

The algebra $K[Y]^{\delta}$ coincides with the algebra $K[Y]^G$ of invariants of the cyclic group $G$ generated by
\[
\exp(\delta)=1+\frac{\delta}{1!}+\frac{\delta^2}{2!}+\cdots
\]
and with the algebra of invariants of the additive group $K_a$ of the field $K$ with its $d$-dimensional representation
\[
\alpha\to\exp(\alpha\delta),\quad \alpha\in K_a,
\]
which allows to involve invariant theory.
Historically, it seems that this relation was used quite rarely and some of the results on
Weitzenb\"ock derivations rediscover classical results in invariant theory. For example, the modern proof of
the theorem of Weitzenb\"ock given by Seshadri \cite{Se} is equivalent to the results of Roberts \cite{Ro}
that for an $SL_2$-module $W$ the algebra $K[W]^{UT_2}$ is isomorphic to the algebra of covariants.
There is an elementary version of the proof by Seshadri given by Tyc \cite{Ty} which is in the language
of representations of the Lie algebra $sl_2(K)$ and can be followed without serious algebraic knowledge.
Let $\delta$ be a Weitzenb\"ock derivation.
All eigenvalues of $\delta$ (acting on $KY$) are equal to 0 and, up to a linear change of the
coordinates of $K[Y]$, $\delta$ is determined by its Jordan normal form. Hence, for each fixed dimension $d$
there is only a finite number of Weitzenb\"ock derivations.
The only derivation which corresponds to a single Jordan cell is called {\it basic}.
Onoda \cite{O} presented an algorithm which calculates the Hilbert series in the case
of a basic Weitzenb\"ock derivation. He calculated the
Hilbert series for the basic derivation $\delta$ and $d=6$
and, as a consequence showed that the algebra of constants  ${\mathbb C}[Y]^{\delta}$ is not
a complete intersection. (By the same paper \cite{O},
${\mathbb C}[Y]^{\delta}$ is Gorenstein for any Weitzenb\"ock derivation
$\delta$ which agrees with a general fact in invariant theory of classical groups.)
Other methods to compute the Hilbert series of $K[W]^{\delta}$ for any Weitzenb\"ock derivation $\delta$
are developed by Bedratyuk; see \cite{Bd5} and the references there.
Below we show how the MacMahon partition analysis can be used to compute the Hilbert series of $K[W]^{\delta}$.
The following theorem and its corollary were announced in \cite{DG2}.

\begin{theorem}\label{Weitzenboeck}
Let $\delta$ be a Weitzenb\"ock derivation of $K[Y]$
with Jordan normal form consisting of $k$ cells
of sizes $d_1+1,\ldots,d_k+1$, respectively. Let
\[
f_{\delta}(x_1,x_2,t)=\frac{1}{q_{d_1}(x_1,x_2,t)\cdots q_{d_k}(x_1,x_2,t)},
\]
where
\[
q_d(x_1,x_2,t)=(1-x_1^dt)(1-x_1^{d-1}x_2t)\cdots (1-x_1x_2^{d-1}t)(1-x_2^dt).
\]
Then the Hilbert series of the algebra of constants $K[Y]^{\delta}$
is given by
\[
H(K[Y]^{\delta};t)=M(f_{\delta};1,1),
\]
where $M(f_{\delta};x_1,x_2)$ is the multiplicity series of the symmetric with respect to $x_1,x_2$
function $f_{\delta}(x_1,x_2,t)\in K(t)[[x_1,x_2]]^{S_2}$.
\end{theorem}

\begin{proof}
If $\delta$ has $k$ Jordan cells and the $i$th cell is of size $d_i+1$, $i=1,\ldots,k$, we identify
the vector space $KY$ with the $GL_2$-module
\[
W=W(d_1)\oplus\cdots\oplus W(d_k)
\]
and the algebra $K[Y]$ with the symmetric algebra $K[W]$. Then the algebra of constants
$K[Y]^{\delta}$ coincides with the algebra $K[W]^{UT_2}$ of $UT_2$-invariants.
Obviously, the function $f_{\delta}(x_1,x_2,t)$ is equal to the Hilbert series of
the $\mathbb Z$-graded $GL_2$-module $K[W]$. Hence Theorem \ref{SL- and UT-invariants}
completes the proof.
\end{proof}

\begin{example}\label{Weitzenboeck for d up to 7}
Let $\delta=\delta(d_1,\ldots,d_k)$ be the Weitzenb\"ock derivation
with $k$ Jordan cells of size $d_i+1$, $i=1,\ldots,k$, respectively.
If the matrix of $\delta$
contains a Jordan cell of size 1 corresponding to $x_d$, then
\[
K[x_1,\ldots,x_{d-1},x_d]^{\delta}=
K[x_1,\ldots,x_{d-1}]^{\delta}[x_d]
\]
and the Hilbert series of the algebras of constants $K[x_1,\ldots,x_d]^{\delta}$ and
$K[x_1,\ldots,x_{d-1}]^{\delta}$ are related by
\[
H(K[x_1,\ldots,x_d]^{\delta};t)=\frac{1}{1-t}H(K[x_1,\ldots,x_{d-1}]^{\delta};t).
\]
Hence it is sufficient to consider only $\delta$ with Jordan matrices
without 1-cells. Below we extend the results in Examples \ref{invariants for homogeneous modules} and
\ref{invariants - nonhomogeneous case} and give the Hilbert series of $K[Y]^{\delta}$
for all possible $\delta$ with $d\leq 7$. Originally the computations were performed in \cite{DG2}
illustrating the methods developed there for symmetric functions in two variables.
Here we repeated the computations with the methods of the MacMahon partition analysis.
Clearly, the results coincide with those from \cite{Bd5}.

\noindent $d=2$:
\[
H(K[Y]^{\delta(1)};t)=\frac{1}{1-t};
\]
$d=3$:
\[
H(K[Y]^{\delta(2)};t)=\frac{1}{(1-t)(1-t^2)};
\]
$d=4$:
\[
H(K[Y]^{\delta(3)};t)=\frac{1-t+t^2}{(1-t)^2(1-t^4)}=\frac{1+t^3}{(1-t)(1-t^2)(1-t^4)},
\]
\[
H(K[Y]^{\delta(1,1)};t)=\frac{1}{(1-z)^2(1-z^2)};
\]
$d=5$:
\[
H(K[Y]^{\delta(4)};t)=\frac{1-t+t^2}{(1-t)^2(1-t^2)(1-t^3)}=\frac{1+t^3}{(1-t)(1-t^2)^2(1-t^3)},
\]
\[
H(K[Y]^{\delta(2,1)};t)
=\frac{1}{(1-z)^2(1-z^2)(1-z^3)};
\]
$d=6$:
\[
H(K[Y]^{\delta(5)};t)=\frac{p(z)}{(1-z)(1-z^2)(1-z^4)(1-z^6)(1-z^8)},
\]
\[
p(z)=1+z^2+3z^3+3z^4+5z^5+4z^6+6z^7+6z^8+4z^9+5z^{10}+3z^{11}+3z^{12}+z^{13}+z^{15},
\]
\[
H(K[Y]^{\delta(3,1)};t)=\frac{1+z^2+3z^3+z^4+z^6}{(1-z)^2(1-z^2)(1-z^4)^2},
\]
\[
H(K[Y]^{\delta(2,2)};t)=\frac{1+z^2}{(1-z)^2(1-z^2)^3},
\]
\[
H(K[Y]^{\delta(1,1,1)};t)
=\frac{1-z^3}{(1-z)^3(1-z^2)^3}=\frac{1+z+z^2}{(1-z)^2(1-z^2)^3};
\]
$d=7$:
\[
H(K[Y]^{\delta(6)};t)=\frac{1+z^2+3z^3+4z^4+4z^5+4z^6+3z^7+z^8+z^{10}}{(1-z)(1-z^2)^2(1-z^3)(1-z^4)(1-z^5)},
\]
\[
H(K[Y]^{\delta(4,1)};t)=\frac{1+2z^2+2z^3+4z^4+2z^5+2z^6+z^8}{(1-z)^2(1-z^2)(1-z^3)^2(1-z^5)},
\]
\[
H(K[Y]^{\delta(3,2)};t)
=\frac{1-z+4z^2-z^3+5z^4-z^5+4z^6-z^7+z^8}{(1-z)^3(1-z^3)(1-z^4)(1-z^5)}
\]
\[
=\frac{1+3z^2+3z^3+4z^4+4z^5+3z^6+3z^7+z^9}{(1-z)^2(1-z^2)(1-z^3)(1-z^4)(1-z^5)},
\]
\[
H(K[Y]^{\delta(2,1,1)};t)
=\frac{1+3z^2+z^4}{(1-z)^3(1-z^2)(1-z^3)^2}.
\]
\end{example}

\begin{corollary}\label{not complete intersection}
For $d\leq 7$ the algebra of constants $K[Y]^{\delta}$ of the Weitzenb\"ock derivation
$\delta=\delta(d_1,\ldots,d_k)$ is not a complete intersection for
\[
(d_1,\ldots,d_k)=(5),(3,1),(6),(4,1),(3,2),(2,1,1).
\]
\end{corollary}

\begin{proof}
Using, as in \cite{O}, that the zeros of the nominator of the Hilbert series of
a complete intersection are roots of unity (see \cite{S2}),
the proof follows immediately from
Example \ref{Weitzenboeck for d up to 7}. (The case $(d_1,\ldots,d_k)=(5)$ was
established in \cite{O}.)
\end{proof}

\section{PI-algebras and noncommutative invariant theory}

In this section we assume that all algebras are unital (and $\text{char}(K)=0$).
For a background on PI-algebras we refer, e.g., to \cite{D2}.
Let $Y_{\infty}=\{y_1,y_2,\ldots\}$ and let $K\langle Y_{\infty}\rangle$
be the free associative algebra of countable rank freely generated by $Y$.
This is the algebra of polynomials in infinitely many noncommutative
variables. Let $K\langle Y\rangle=K\langle y_1,\ldots,y_d\rangle$
be its subalgebra of rank $d$. Recall that $f(y_1,\ldots,y_m)\in K\langle Y_{\infty}\rangle$
is called a polynomial identity for the associative algebra $R$ if
$f(r_1,\ldots,r_m)=0$ for all $r_1,\ldots,r_m\in R$. If $R$ satisfies a nonzero polynomial identity,
it is called a {\it PI-algebra}. We denote by $T_{\infty}(R)$ the ideal of all polynomial identities
of $R$ (called the {\it T-ideal of} $R$) and
\[
T(R)=K\langle Y\rangle\cap T_{\infty}(R)
\]
is the T-ideal of the polynomial identities in $d$ variables for $R$.
Since we work over a field of characteristic 0, all polynomial identities of $R$
follow from the multilinear ones. The vector space of the multilinear polynomials
of degree $n$
\[
P_n=\text{span}\{y_{\sigma(1)}\cdots y_{\sigma(n)}\mid \sigma\in S_n\}\subset K\langle Y_{\infty}\rangle
\]
has a natural structure of a left $S_n$-module and the factor space
\[
P_n(R)=P_n/(P_n\cap T_{\infty}(R))
\]
is its $S_n$-factor module. One of the main problems in the quantitative study of PI-algebras
is to compute the {\it cocharacter sequence of} $R$
\[
\chi_n(R)=\chi_{S_n}(P_n(R))=\sum_{\lambda\vdash n}m_{\lambda}(R)\chi_{\lambda},
\]
where $\chi_{\lambda}$, $\lambda\vdash n$, is the irreducible $S_n$-character indexed by
the partition $\lambda$. A possible way to compute the multiplicities $m_{\lambda}(R)$ is the following.
One considers the diagonal $GL_d$-action on $K\langle Y\rangle$ extending the natural action of $GL_d$
on the $d$-dimensional vector space $KY$ with basis $Y$. Then the factor algebra
\[
F(R)=K\langle Y\rangle/T(R)
\]
called the {\it relatively free algebra of rank $d$ in the variety of associative algebras generated by} $R$,
inherits the $GL_d$-action of $K\langle Y\rangle$. Its Hilbert series as a $GL_d$-module coincides with
its Hilbert series as a ${\mathbb Z}^d$-graded vector space with grading defined by
\[
\deg(y_i)=(\underbrace{0,\ldots,0}_{i-1 \text{ times}},1,\underbrace{0,\ldots,0}_{d-i \text{ times}}).
\]
It is a symmetric function in $d$ variables and
\[
H(F(R);X)=\sum m_{\lambda}(R)S_{\lambda}(X),
\]
where the sum is on all $(\lambda_1,\ldots,\lambda_d)$ and the multiplicities $m_{\lambda}(R)$ are exactly the same
as in the cocharacter sequence of $R$. Hence, if we know the Hilbert series of $F(R)$, we can compute the multiplicities
$m_{\lambda}(R)$ for partitions $\lambda$ in $\leq d$ parts. The theorem of Belov \cite{Be} gives that for any
PI-algebra $R$ the Hilbert series of $F(R)$ is a rational function. Berele \cite{B2} found that the proof of Belov
also implies that this Hilbert series is a nice rational symmetric function. Hence we can apply our methods to
calculate the multiplicity series of $H(F(R);X)$ and to find the multiplicities of $R$.
See the introduction of Boumova and Drensky \cite{BD} for a survey of results on the multiplicities of concrete algebras.

The most important algebras in PI-theory are the so called {\it T-prime algebras}
whose T-ideals are the building blocks of the structure theory of T-ideals developed by Kemer,
see his book \cite{Ke} for the account. There are few cases only when the Hilbert series of
the relatively free algebras $F(R)$ are explicitly known.
For T-prime algebras these are the base filed $K$,
the Grassmann (or exterior) algebra $E$, the $2\times 2$ matrix algebra $M_2(K)$, and the algebra
$M_{1,1}\subset M_2(E)$ which has the same polynomial identities as the tensor square $E\otimes_KE$ of the Grassmann algebra.
In all these cases the multiplicities are also known.
The case $R=K$ is trivial because $F(K)=K[Y]$:
\[
m_{\lambda}(K)
=\begin{cases}
1,\text{ if }\lambda=(n),\\
0,\text{ otherwise.}\\
\end{cases}
\]
The multiplicities for $M_2(K)$ were obtained by Formanek \cite{F1} and Drensky \cite{D1}, see also \cite{D2}:
\[
m_{\lambda}(M_2(K))=
\begin{cases}
0,\text{ if } \lambda_5>0,\\
1,\text{ if }\lambda=(n),\\
(\lambda_1-\lambda_2+1)\lambda_2,\text{ if } \lambda=(\lambda_1,\lambda_2),\lambda_2>0,\\
\lambda_1(2-\lambda_4)-1,\text{ if } \lambda=(\lambda_1,1,1,\lambda_4),\\
(\lambda_1-\lambda_2+1)(\lambda_2-\lambda_3+1)(\lambda_3-\lambda_4+1)
\text{ in all other cases.}\\
\end{cases}
\]
Hence the multiplicity series are
\begin{eqnarray*}
M(H(F(K));X)&=&\frac{1}{1-x_1},
\\
 M'(H(F(K));V)&=&\frac{1}{1-v_1};
\\
M'(H(F(M_2(K)));V)&=&\frac{1}{(1-v_1)^2(1-v_2)^2(1-v_3)^2(1-v_4)}
\\
&&\mbox{}-\frac{v_2+v_1(1-v_2)}{(1-v_1)^2(1-v_2)}-\frac{v_3+v_4}{1-v_1}.
\end{eqnarray*}
For any $d$, there are partitions $\lambda=(\lambda_1,\ldots,\lambda_d)$ with $\lambda_d>0$
and nonzero multiplicities $m_{\lambda}(E)$ and $m_{\lambda}(E\otimes_K E)$. Hence for these cases
the multiplicity series $M(H(E);X)$ and $M(H(E\otimes_KE);X)$ do not carry all the information about
the cocharacter sequences of $E$ and $E\otimes_KE$. For this purpose Berele \cite{B3} suggested to use
hook Schur functions instead of ordinary ones.

Another case when the Hilbert series of the relatively free algebras may be computed and used
to find the multiplicities is for algebras $R$ with T-ideals which
are products of two T-ideals, $T(R)=T(R_1)T(R_2)$. See again \cite{BD} for details.
Formanek \cite{F2} found the following simple formula for the Hilbert
series of $T(R)$ in terms of the Hilbert series of $T(R_1)$ and $T(R_2)$:
\[
H(T(R))=\frac{H(T(R_1))H(T(R_2))}{H(K\langle Y\rangle)}=(1-(x_1+\cdots+x_d))H(T(R_1))H(T(R_2)).
\]
Translated for the corresponding relatively free algebras this gives
\[
H(F(R))=H(F(R_1))+H(F(R_2))+((x_1+\cdots+x_d)-1)H(F(R_1))H(F(R_2)).
\]
It is known that $T(U_m(K))=T^m(K)$ (Maltsev \cite{Ma}) and $T(U_m(E))=T^m(E)$ (this follows from the results of Abakarov \cite{Ab}),
where $U_k(K)$ and $U_k(E)$ are the algebras of $k\times k$ upper triangular matrices
with entries from $K$ and $E$, respectively. The multiplicities of $U_k(K)$ were studied
by Boumova and Drensky \cite{BD}, with explicit results for ``large'' partitions
$\lambda=(\lambda_1,\ldots,\lambda_n)$
(such that $\lambda_{k+1}+\cdots+\lambda_n=k-1$). The multiplicities of $U_2(E)$
were determined by Centrone \cite{Ce}. In both cases the results were obtained
using the Young rule only, without the MacMahon partition analysis.
Here we shall illustrate once again Algorithm \ref{algorithm of Xin}.

\begin{example}\label{metabelian algebras}
Let $C$ be the commutator ideal of the free associative algebra $K\langle Y\rangle$.
Then by Maltsev \cite{Ma} the T-ideal $C^k$ coincides with the T-ideal of $U_k(K)$. Since
$K\langle Y\rangle/C$ is the polynomial algebra in $d$ variables and
\[
H(K[Y];X)=\prod_{i=1}^d\frac{1}{1-x_i}=\sum_{n\geq 0}S_{(n)}(X),
\]
the formula
\[
H(F(R))=H(F(R_1))+H(F(R_2))+((x_1+\cdots+x_d)-1)H(F(R_1))H(F(R_2)).
\]
for the Hilbert series of relatively free algebras corresponding to products of T-ideals gives
\[
H(F(U_2(K));X)=2\prod_{i=1}^d\frac{1}{1-x_i}+((x_1+\cdots+x_d)-1)\prod_{i=1}^d\frac{1}{(1-x_i)^2}.
\]
The decomposition of the product $S_{\mu}(X)S_{(n)}(X)=\sum S_{\lambda}(X)$ is given by the Young rule.
If $\mu$ is a partition in $k$ parts, then $\lambda$ is a partition in $k$ or $k+1$ parts. Hence
$H(F(U_2(K));X)$ decomposes into a series of Schur functions $S_{\lambda}(X)$, where $\lambda$ is a partition
in no more than three parts. Therefore, it is sufficient to consider the multiplicity series of $H(F(U_2(K));X)$
for $d=3$ only. Clearly,
\[
M'(H(K[Y]);V)=M'\left(\prod_{i=1}^d\frac{1}{1-x_i};V\right)=\frac{1}{1-v_1}.
\]
Algorithm \ref{algorithm of Xin} gives
\[
g_1(x_1,x_2,x_3)=\frac{(x_1-x_2)(x_1-x_3)(x_2-x_3)((x_1+x_2+x_3)-1)}{(1-x_1)^2(1-x_2)^2(1-x_3)^2},
\]
\[
g_1(x_1z_1,x_2/z_1,x_3)
=\frac{-x_3+x_1x_2-2x_1x_2x_3-x_1x_3z_1}{(1-x_3)^2}
\]
\[
+\frac{x_3+x_3^2-x_3^3-x_1x_2+x_1x_2x_3-x_1^2x_2^2x_3}{(1-x_1x_2)(1-x_3)^2(1-x_1z_1)}
+\frac{-x_3^2+x_1^2x_2^2}{(1-x_1x_2)(1-x_3)(1-x_1z_1)^2}
\]
\[
+\frac{x_2x_3}{(1-x_3)^2z_1}
+\frac{(x_3-x_3^2+x_3^3-x_1x_2+x_1x_2x_3+2x_1^2x_2^2-3x_1^2x_2^2x_3)x_2}{(1-x_1x_2)(1-x_3)^2(x_2-z_1)}
\]
\[
+\frac{(x_3^2-x_1^2x_2^2)x_2^2}{(1-x_1x_2)(1-x_3)(x_2-z_1)^2}.
\]
We omit the last three summands which give negative degrees of $z_1$ in the expansion of
$g_1(x_1z_1,x_2/z_1,x_3)$ as a Laurent series and, substituting $z_1=1$, we obtain
\[
g_2(x_1,x_2,x_3)
=\frac{-x_3+x_1x_2-2x_1x_2x_3-x_1x_3}{(1-x_3)^2}
\]
\[
+\frac{x_3+x_3^2-x_3^3-x_1x_2+x_1x_2x_3-x_1^2x_2^2x_3}{(1-x_1)(1-x_1x_2)(1-x_3)^2}
+\frac{-x_3^2+x_1^2x_2^2}{(1-x_1)^2(1-x_1x_2)(1-x_3)}.
\]
Repeating the procedure with $g_2(x_1,x_2z_2,x_3/z_2)$ we obtain
\[
\mathop{\Omega}\limits_{\geq}(g_1(x_1z_1,x_2z_2/z_1,x_3/z_2))
=\frac{x_1^2x_2(-1+x_1+2x_1x_2-x_1^2x_2+x_1x_2x_3)}{(1-x_1)^2(1-x_1x_2)},
\]
\[
M'\left(\frac{(x_1+x_2+x_3)-1}{(1-x_1)^2(1-x_2)^2(1-x_3)^2};V\right)
=\frac{-1+v_1+2v_2-v_1v_2+v_3}{(1-v_1)^2(1-v_2)}
\]
\[
=-\frac{1}{1-v_1}+\frac{v_2+v_3}{(1-v_1)^(1-v_2)},
\]
\[
M'(H(F(U_2(K));V)=\frac{1}{1-v_1}+\frac{v_2+v_3}{(1-v_1)^(1-v_2)}
\]
\[
=\sum_{n\geq 0}v_1^n+\sum_{p\geq 0}\sum_{q\geq 1}(p+1)v_1^pv_2^q+\sum_{p\geq 0}\sum_{q\geq 0}(p+1)v_1^pv_2^qv_3.
\]
Hence the multiplicities in the cocharacter sequence of $U_2(K)$ are
\[
m_{\lambda}(U_2(K))=\begin{cases}
1,\text{ if }\lambda=(\lambda_1),\\
\lambda_1-\lambda_2+1,\text{ if }\lambda=(\lambda_1,\lambda_2),\lambda_2>0,\\
\lambda_1-\lambda_2+1,\text{ if }\lambda=(\lambda_1,\lambda_2,1),\\
0 \text{ in all other cases.}\\
\end{cases}
\]
Compare our approach with the approach on the multiplicities of $U_2(K)$ given by Mishchenko, Regev, and Zaicev \cite{MRZ}
and in \cite{BD}.
\end{example}

A case of products of T-ideals when we do need the MacMahon partition analysis
is of block triangular matrices with entries from the field.
Let $d_1,\ldots,d_m$ be positive integers and let $U(d_1,\ldots,d_m)$
be the algebra of matrices of the form
\[
\left(
\begin{matrix}
M_{d_1}(K) &\ast & \dots &\ast & \ast \\
0 & M_{d_2}(K) &\dots &\ast &\ast \\
\vdots & \vdots& \ddots &\vdots & \vdots \\
0&0&\dots&M_{d_{m-1}}(K)&\ast\\
0 &0& \dots & 0 & M_{d_m}(K)\\
\end{matrix}
\right).
\]
It is known, see Giambruno and Zaicev \cite{GZ}, that
\[
T(U(d_1,\ldots,d_m))=T(M_{d_1}(K))\cdots T(M_{d_m}(K)).
\]
The only cases when we know the Hilbert series of $T(M_k(K))$ are $k=1,2$,
and we can compute the Hilbert series of $F(U(d_1,\ldots,d_m))$.
The multiplicities of $U(d_1,\ldots,d_m)$ when all $d_i$ are equal to 1 and 2
were studied in the master thesis of Kostadinov \cite{K}, see also his paper with Drensky \cite{DK}.
If $d_1=\cdots=d_m=1$, the algebra $U(d_1,\ldots,d_m)$ is equal to $U_k(K)$, handled in \cite{BD}.
If only one $d_i$ is equal to 2 and the others are equal to 1, we still can use the Young rule.
The MacMahon partition analysis was applied in \cite{K} in the case when several $d_i$ are equal to 2.
In particular exact formulas for the multiplicity series and the multiplicities as well as the asymptotics
of the multiplicities were found for a small number of blocks.

Studying the polynomial identities of the matrix algebra $M_k(K)$, there is another object which behaves
much better than the relatively free algebra $F(M_k(K))$. Let
\[
K[Z]=K[z_{pq}^{(i)}\mid p,q=1,\ldots,k, i=1,\ldots,d]
\]
be the polynomial algebra in $k^2d$ commuting variables and let $R_{kd}$ be {\it the generic matrix algebra}
generated by the $d$ generic $k\times k$ matrices
\[
z_i=\left(\begin{matrix}
z_{11}^{(i)}&\cdots&z_{1k}^{(i)}\\
\vdots&\ddots&\vdots\\
z_{k1}^{(i)}&\cdots&z_{kk}^{(i)}\\
\end{matrix}\right),
\quad i=1,\ldots,d.
\]
It is well known that $R_{kd}\cong F(M_k(K))$. Let $C_{kd}$ be {\it the pure} (or {\it commutative}) {\it trace algebra}
generated by all traces of products $\text{tr}(z_{i_1}\cdots z_{i_n})$, $i_1,\ldots,i_n=1,\ldots,d$.
It coincides with the algebra of invariants $K[Z]^{GL_k}$ where the action of $GL_k$ on $K[Z]$ is induced
by the action of $GL_k$ on the generic matrices $z_1,\ldots,z_d$
by simultaneous conjugation. Hence one may study $C_{kd}$ with methods of the classical invariant theory.
{\it The mixed} (or {\it noncommutative}) {\it trace algebra} $T_{kd}=C_{kd}R_{kd}$ also has a meaning in invariant theory.
See the books by Formanek \cite{F4}, and also with Drensky \cite{DF}, for a background on trace algebras.
The mixed trace algebra approximates quite well the algebra $F(M_k(K))$. In particular,
one may consider the multilinear components of the pure and mixed trace algebras $C_k=C_{k,\infty}$ and $T_k=T_{k,\infty}$
of infinitely many generic $k\times k$ matrices and the related sequences
\[
\chi_n(C_k)=\sum_{\lambda\vdash n}m_{\lambda}(C_k)\chi_{\lambda},\quad
\chi_n(T_k)=\sum_{\lambda\vdash n}m_{\lambda}(T_k)\chi_{\lambda},\quad n=0,1,2,\ldots,
\]
of $S_n$-characters called {\it the pure and mixed
cocharacter sequences}, respectively.
Formanek \cite{F3} showed that the multiplicities $m_{\lambda}(T_k)$ in the mixed cocharacter sequence
and $m_{\lambda}(M_k(K))$ in the ordinary cocharacter sequence
for $M_k(K)$ coincide for all partitions $\lambda=(\lambda_1,\ldots,\lambda_{k^2})$ with $\lambda_{k^2}\geq 2$.
The only case when the pure and mixed cocharacter sequences are known is for $n=2$
due to Procesi \cite{P} and Formanek \cite{F1} (besides the trivial case of $k=1$). We state the result for $T_2$ only.
\[
m_{\lambda}(T_2)=\begin{cases}
(\lambda_1-\lambda_2+1)(\lambda_2-\lambda_3+1)(\lambda_3-\lambda_4+1), \text{ if }\lambda=(\lambda_1,\lambda_2,\lambda_3,\lambda_4),\\
0\text{ otherwise,}
\end{cases}
\]
The situation with the Hilbert series of $C_{kd}$ and $T_{kd}$ is better. The case $k=2$ was handled by Procesi \cite{P} and Formanek \cite{F1}:
\begin{eqnarray*}
H(C_{2d};X)&=&\prod_{i=1}^d\frac{1}{1-x_i}\sum_{p,q,r\geq 0}S_{(2p+2q+r,2q+r,r)}(X),
\\
H(T_{2d};X)&=&\prod_{i=1}^d\frac{1}{1-x_i}\sum_{(\lambda_1,\lambda_2,\lambda_3)}S_{(\lambda_1,\lambda_2,\lambda_3)}(X)
\\
&=&\prod_{i=1}^d\frac{1}{(1-x_i)^2}\sum_{n\geq 0}S_{(n,n)}(X).
\end{eqnarray*}
The Molien--Weyl formula gives that the Hilbert series
of $C_{kd}$ and $T_{kd}$ can be expressed as multiple integrals but for
$k\geq 3$ their direct evaluation is quite difficult and was performed
by Teranishi \cite{T1, T2} for $C_{32}$ and $C_{42}$ only.
Van den Bergh \cite{VdB} found a graph-theoretical approach for the
calculation of $H(C_{kd})$ and $H(T_{kd})$.
Berele and Stembridge \cite{BS} calculated the Hilbert series of $C_{kd}$ and $T_{kd}$ for
$k=3$, $d\leq 3$ and of $T_{42}$, correcting also some
typographical errors in the expression of $H(C_{42})$ in \cite{T2}.
Recently  Djokovi\'c \cite{Dj2} computed the Hilbert series
of $C_{k2}$ and $T_{k2}$ for $k=5$ and 6.

Using the Hilbert series of $C_{32}$, Berele
\cite{B1} found an asymptotic expression of
$m_{(\lambda_1,\lambda_2)}(C_3)$. The explicit form of multiplicity series
of the Hilbert series of $C_{32}$ was found by
Drensky and Genov \cite{DG1} correcting also a minor technical error (a missing summand) in \cite{B1}.
The quite technical method of \cite{DG1} was improved in \cite{DG2} and applied by Drensky, Genov and Valenti \cite{DGV}
to compute the multiplicity series of $H(T_{32})$ and by Drensky and Genov \cite{DG3} for the multiplicity series
of $H(C_{42})$ and $H(T_{42})$. In principle, the same methods work for the multiplicities of the Hilbert series of
$H(C_{k2})$ and $H(T_{k2})$, $k=5,6$.

\begin{example}\label{multiplicities of T32}
We shall apply the MacMahon partition analysis to find the multiplicities of $H(T_{32})$.
By Berele and Stembridge \cite{BS}
\[
H(T_{32},x_1,x_2)
=\frac{1}{(1-x_1)^2(1-x_2)^2(1-x_1^2)(1-x_2^2)(1-x_1x_2)^2(1-x_1^2x_2)(1-x_1x_2^2)}.
\]
As in Example \ref{multiplicities of KW3 from DG1}
we define the function
\[
g(x_1,x_2)=(x_1-x_2)H(T_{32};x_1,x_2)
\]
and decompose $g(x_1z,x_2/z)$ as a sum of partial fractions with respect to $z$. The result is
\[
\frac{1}{2(1-x_1x_2)^6(1+x_1x_2)^2(1-x_1z)^3}
+\frac{1+2x_1x_2-5x_1^2x_2^2}{4(1-x_1x_2)^7(1+x_1x_2)^3(1-x_1z)^2}
\]
\[
-\frac{1+2x_1x_2-10x_1^2x_2^2+10x_1^3x_2^3-7x_1^4x_2^4}{8(1-x_1x_2)^8(1+x_1x_2)^4(1-x_1z)}
\]
\[
-\frac{1}{8(1-x_1x_2)^2(1+x_1x_2)^4(1+x_1^2x_2^2)(1+x_1z)}
\]
\[
-\frac{x_1^3x_2^3}{(1-x_1x_2)^8(1+x_1x_2)^4(1+x_1^2x_2^2)(1-x_1^2x_2z)}
+\frac{x_2^3}{2(1-x_1x_2)^6(1+x_1x_2)^2(x_2-z)^3}
\]
\[
+\frac{x_2^2(-5+2x_1x_2+x_1^2x_2^2)}{4(1-x_1x_2)^7(1+x_1x_2)^3(x_2-z)^2}
+\frac{x_2(7-10x_1x_2+10x_1^2x_2^2-2x_1^3x_2^3-x_1^4x_2^4)}{8(1-x_1x_2)^8(1+x_1x_2)^4(x_2-z)}
\]
\[
-\frac{x_2}{8(1-x_1x_2)^2(1+x_1^2x_2^2)(1+x_1x_2)^4(x_2+z)}
\]
\[
-\frac{x_1^4x_2^5}{(1-x_1x_2)^8(1+x_1x_2)^4(1+x_1^2x_2^2)(x_1x_2^2-z)}.
\]
We remove the last five summands because their expansions as Laurent series contain
negative degrees of $z$ only. Then we replace $z$ by 1 and obtain
\[
\mathop{\Omega}\limits_{\geq}(g(x_1z,x_2/z))=
\frac{1}{2(1-x_1x_2)^6(1+x_1x_2)^2(1-x_1)^3}
\]
\[
+\frac{1+2x_1x_2-5x_1^2x_2^2}{4(1-x_1x_2)^7(1+x_1x_2)^3(1-x_1)^2}
-\frac{1+2x_1x_2-10x_1^2x_2^2+10x_1^3x_2^3-7x_1^4x_2^4}{8(1-x_1x_2)^8(1+x_1x_2)^4(1-x_1)}
\]
\[
-\frac{1}{8(1-x_1x_2)^2(1+x_1x_2)^4(1+x_1^2x_2^2)(1+x_1)}
\]
\[
-\frac{x_1^3x_2^3}{(1-x_1x_2)^8(1+x_1x_2)^4(1+x_1^2x_2^2)(1-x_1^2x_2)}.
\]
Dividing $\mathop{\Omega}\limits_{\geq}(g(x_1z,x_2/z))$ by $x_1$ and after the substitution $v_1=x_1$, $v_2=x_1x_2$ we obtain
\[
M'(H(T_{32});v_1,v_2)=\frac{v_1^3h_3(v_2)+v_1^2h_2(v_2)+v_1h_1(v_2)+h_0(v_2)}{(1-v_1)^3(1+v_1)(1-v_1v_2)(1-v_2)^7(1+v_2)^4(1+v_2^2)},
\]
\[
h_3(v_2)=v_2^2(v_2^4-v_2^3+3v_2^2-v_2+1),\quad h_2(v_2)= v_2(2v_2^4-4v_2^3+v_2^2-v_2-1),
\]
\[
h_1(v_2)=v_2(-v_2^4-v_2^3+v_2^2-4v_2+2), \quad h_0(v_2)=v_2^4-v_2^3+3v_2^2-v_2+1,
\]
which coincides with the result of \cite{DGV}. The multiplicity series has also the form
\[
M'(H(T_{32});v_1,v_2)=\frac{a_3(v_2)}{(1-t)^3}+\frac{a_2(v_2)}{(1-t)^2}
+\frac{a_1(v_2)}{(1-t)}+\frac{b(v_2)}{1+t}+\frac{c(v_2)}{1-v_2t},
\]
where
\[
a_3(v_2)=\frac{1}{2(1-v_2)^6(1+v_2)^2},\quad
a_2(v_2)=\frac{(3v_2^2-2v_2+1)}{2^2(1-v_2)^7(1+v_2)^3},
\]
\[
a_1(v_2)=\frac{(v_2^4-6v_2^3+14v_2^2-6v_2+1)} {2^3(1-v_2)^8(1+v_2)^4},
\]
\[
b(v_2)=\frac{1}{2^3(1-v_2)^2(1+v_2)^4(1+v_2^2)}, \quad
c(v_2)=\frac{-v_2^4}{(1-v_2)^8(1+v_2)^4(1+v_2^2)}.
\]
\end{example}

In a forthcoming paper by Benanti, Boumova and Drensky \cite{BBD}  we shall apply our methods
to find the multiplicities in the pure and mixed cocharacter sequence of three generic $3\times 3$ matrices.

One of the directions of noncommutative invariant theory is to study subalgebras of invariants
of linear groups acting on free and relatively free algebras. For a background see the surveys by
Formanek \cite{F2} and Drensky \cite{D3}. Recall that we consider the action of $GL_d$ on
the vector space $KY$ with basis $Y=\{y_1,\ldots,y_d\}$ and extend this action diagonally on
the free algebra $K\langle Y\rangle$ and the relatively free algebras $F(R)$, where $R$ is a PI-algebra.
Let $G$ be a subgroup of $GL_d$.
Then $G$ acts on $F(R)$ and the algebra of $G$-invariants is
\[
F(R)^G=\{f(Y)\in F(R)\mid g(f)=f\text{ for all }g\in G\}.
\]
Comparing with commutative invariant theory, when $K[Y]^G$ is finitely generated for all ``nice'' groups (e.g., finite
and reductive), the main difference in the noncommutative case is that $F(R)^G$ is finitely generated quite rarely.
For a survey on invariants of finite groups $G$ see \cite{D3, F2} and the survey by Kharlampovich and Sapir \cite{KS}.
For a fixed PI-algebra $R$ there are many conditions
which are equivalent  to the fact that the algebra $F(R)^G$ is finitely generated for all finite groups $G$.
Maybe the simplest one is that this happens if and only if
$F(R)^G$ is finitely generated for $d=2$ and the cyclic group $G=\langle g\rangle$ of order 2 generated by the matrix
\[
g=\left(\begin{matrix}
-1&0\\
0&1\\
\end{matrix}\right).
\]
Clearly, in this case $F(R)^{\langle g\rangle}$ is spanned on all monomials in $y_1,y_2$ which are of even degree
with respect to $y_1$. For the algebras of invariants $F(R)^G$ of reductive groups $G$ see Vonessen \cite{V}
and Domokos and Drensky \cite{DD}.

Concerning the Hilbert series of $F(R)^G$, for $G$ finite there is an analogue of the Molien formula, see Formanek \cite{F2}:
If $\xi_1(g),\ldots,\xi_d(g)$ are the eigenvalues of $g \in G$, then
the Hilbert series of the algebra of invariants $F(R)^G$ is
\[
H(F(R)^G;t) = \frac{1}{\vert G\vert}\sum_{g\in G}H(F(R);\xi_1(g)t,\ldots,\xi_d(g)t).
\]
Combined with the theorem of Belov \cite{Be} for the rationality of $H(F(R);X)$ (as specified by Berele \cite{B2})
this gives that the Hilbert series of
$F(R)^G$ is a nice rational function for every finite group $G$.
By a result of Domokos and Drensky \cite{DD} the Hilbert series of $F(R)^G$ for a reductive group $G$
is a nice rational function if $R$ satisfies a nonmatrix polynomial identity
(i.e., an identity which does not hold for the algebra $M_2(K)$ of $2\times 2$ matrices).
The proof uses that for algebras $R$ with nonmatrix identity the relatively free algebra has a finite series of graded ideals
with factors which are finitely generated modules of polynomial algebras. This allows to reduce the considerations to the commutative case
when the rationality of the Hilbert series is well known. We believe that the careful study of the proof of Belov \cite{Be}
would give that $H(F(R)^G;t)$ is a nice rational function for every reductive group $G$ and an arbitrary PI-algebra $R$.

Let $W$ be a $p$-dimensional $GL_d$-module with basis $Y_p=\{y_1,\ldots,y_p\}$. Consider the related representation
$\rho:GL_d\to GL_p$ of $GL_d$ in the $p$-dimensional vector space with this basis. If $F_p(R)$ is a relatively free algebra of rank
$p$ freely generated by $Y_p$, then the representation $\rho$ induces an action of $GL_d$ on $F_p(R)$.
The following theorem is a noncommutative analogue of Theorem \ref{SL- and UT-invariants}.

\begin{theorem}\label{noncommutative SL- and UT-invariants}
Let $W$ be a $p$-dimensional polynomial $GL_d$-module
with Hilbert series with respect to the grading
induced by the $GL_d$-action on $W$
\[
H(W;X)=\sum a_ix_1^{i_1}\cdots x_d^{i_d},
\quad a_i\geq 0,a_i\in\mathbb Z,\sum a_i=p.
\]
Let $R$ be a PI-algebra with the corresponding relatively free algebra $F_p(R)$ of rank $p$ freely generated by $Y_p$,
with the natural structure of a $GL_d$-module induced by the $GL_d$-action on $W$.
Let
\[
f(X,t)=H(F_p(R);X^{i^{(1)}}t,\ldots,X^{i^{(p)}}t)
\]
be the formal power series obtained from
the Hilbert series $H(F_p(R);x_1,\ldots,x_p)$ of $F_p(R)$
by substitution of the variables
$x_j$ with $x_1^{i_1}\cdots x_d^{i_d}t$ in such a way that each $x_1^{i_1}\cdots x_d^{i_d}t$ appears exactly $a_i$ times.
Then the Hilbert series of the algebras $F_p(R)^{SL_d}$ and $F_p(R)^{UT_d}$ of $SL_d$- and $UT_d$-invariants are,
respectively,
\[
H(F_p(R)^{SL_d};t)=M'(f;0,\ldots,0,1,t),
\]
\[
H(F_p(R)^{UT_d};t)=M(f;1,\ldots,1,t)=M'(f;1,\ldots,1,t),
\]
where $M(f;X,t)$ and $M'(f;V,t)$ are the multiplicity series of the symmetric in $X$ function $f(X,t)$.
\end{theorem}

\begin{proof}
We may choose the basis $Y_p$ of $W$ to consist of eigenvectors of the diagonal group $D_d$. Then
for a fixed $d$-tuple $i=(i_1,\ldots,i_d)$ exactly $a_i$ of the elements $y_j$ satisfy
\[
g(y_j)=\xi_1^{i_1}\cdots\xi_d^{i_d}y_j,\quad g=\text{diag}(\xi_1,\ldots,\xi_d)\in D_d.
\]
The monomials in $y_1,\ldots,y_p$ are eigenvectors of $D_d$ and $H(F_p(R);X^{i^{(1)}}t,\ldots,X^{i^{(p)}}t)$
is the Hilbert series of the $GL_d$-module $F_p(R)$ which counts also the $\mathbb Z$-grading of $F_p(R)$.
Now the proof is completed as the proof of Theorem \ref{SL- and UT-invariants} because the irreducible $GL_d$-submodule
$W(\lambda)$ contains a one-dimensional $SL_d$-invariant if $\lambda_1=\cdots=\lambda_d$ and does not contain any $GL_d$-invariants
otherwise. Similarly, $W(\lambda)$ contains a one-dimensional $UT_d$-invariant for every $\lambda$.
\end{proof}

Combined with the nice rationality of the Hilbert series of relatively free algebras
Theorem \ref{noncommutative SL- and UT-invariants} immediately gives:

\begin{corollary}\label{nice rational Hilbert series for invariants}
Let $W$ be a $p$-dimensional polynomial $GL_d$-module with basis $Y_p=\{y_1,\ldots,y_p\}$
and let $F_p(R)$ be the relatively free algebra freely generated by $Y_p$ and related to the PI-algebra $R$.
Then the Hilbert series of the algebras of invariants $H(F_p(R)^{SL_d};t)$ and $H(F_p(R)^{UT_d};t)$
are nice rational functions.
\end{corollary}

\begin{example}\label{metabelian invariants}
We shall apply Theorem \ref{noncommutative SL- and UT-invariants}
to Example \ref{metabelian algebras}.
Let $d\geq 2$ and let $SL_d$ and $UT_d$ act as subgroups of $GL_d$ on the relatively free algebra $F(U_2(K))$ with $d$ generators.
Then the generators $y_i\in Y$ of $F(U_2(K))$ are of first degree with respect to the $GL_d$-action.
Hence
\begin{eqnarray*}
H(F(U_2(K));X,t)&=&H(F(U_2(K));Xt)=H(F(U_2(K));x_1t,\ldots,x_dt),
\\
M'(H(F(U_2(K));X,t);V,t)&=&M'(H(F(U_2(K));X,t);v_1t,v_2t^2,\ldots,v_dt^d)
\\
&=&\mbox{}\frac{1}{1-tv_1}+\frac{v_2t^2+v_3t^3}{(1-v_1t)^2(1-v_2)}.
\end{eqnarray*}
And therefore
\begin{eqnarray*}
H(F(U_2(K))^{SL_d};t)&=&M'(H(F(U_2(K));X,t);0,\ldots,0,t^d)
\\
&=&\begin{cases}
\displaystyle{\frac{1}{1-t^2}},\text{ if }d=2,\\
1+t^3,\text{ if }d=3,\\
1,\text{ if }d>3;\\
\end{cases}
\\
H(F(U_2(K))^{UT_d};t)&=&M'(H(F(U_2(K));X,t);t,t^2,\ldots,t^d)
\\
&=&\begin{cases}
\displaystyle{\frac{1-t+t^3}{(1-t)^2(1-t^2)}},\text{ if }d=2,\\
\\
\displaystyle{\frac{1-2t+2t^2}{(1-t)^3}},\text{ if }d\geq 3.\\
\end{cases}
\end{eqnarray*}
\end{example}

\begin{example}\label{nonlinear metabelian}
Again, let $R=U_2(K)$ and let $W=W(1^2)$ be the irreducible $GL_3$-module indexed by the partition $(1^2)=(1,1,0)$.
We consider the relatively free algebra $F_3(U_2(K))$ with the $GL_3$-action induced by the action on $W$.
The Hilbert series of $F_3(U_2(K))$ which counts both the action of $GL_3$ and the $\mathbb Z$-grading is
\begin{eqnarray*}
f(X,t)&=&H(F_3(U_2(K));x_1x_2t,x_2x_3t,x_2x_3t)
\\
&=&\frac{1}{(1-x_1x_2t)(1-x_1x_3t)(1-x_2x_3t)}
\\
&&\mbox{}+\frac{x_1x_2+x_1x_3+x_2x_3-1}{(1-x_1x_2t)^2(1-x_1x_3t)^2(1-x_2x_3t)^2}.
\end{eqnarray*}
Applying Theorem \ref{noncommutative SL- and UT-invariants} and Algorithm \ref{algorithm of Xin} we obtain
\[
M'(f;V,t)=\frac{1-v_2t+(v_1v_2+v_3)v_3t^3}{(1-v_2t)^2(1-v_1v_3t^2)}.
\]
Hence
\begin{eqnarray*}
H(F_3(U_2(K))^{SL_3};t)&=&M'(f;0,0,1,t)=1+t^3,
\\
H(F_3(U_2(K))^{UT_3};t)&=&M'(f;1,1,1,t)=\frac{1-2t+2t^2}{(1-t)^3}.
\end{eqnarray*}
Similarly, if we consider the $GL_3$-module $W=W(2)$, then $GL_3$ acts on $F_6(U_2(K))$ extending the action on $W$,
\[
f(X,t)=H(F_6(U_2(K));x_1^2t,x_2^2t,x_3^2t,x_1x_2t,x_2x_3t,x_2x_3t)
\]
and, applying again Theorem \ref{noncommutative SL- and UT-invariants} and Algorithm \ref{algorithm of Xin}, we obtain
\begin{eqnarray*}
H(F_6(U_2(K))^{ST_3};t)&=&\frac{1-3t^3+6t^6-2t^9}{(1-t^3)^4},
\\
H(F_6(U_2(K))^{UT_3};t)
&=&\frac{p(t)}{((1-t)(1-t^2)(1-t^3))^3},
\end{eqnarray*}
where
\[
p(t)=1-2t+7t^3+11t^4+6t^5-10t^6+t^7+6t^8+4t^9-2t^{10}-4t^{11}+2t^{12}.
\]
\end{example}

\begin{example}\label{invariants of trace algebras}
Let $R_{2p}$ be the algebra generated by $p$ generic $2\times 2$ matrices $z_1,\ldots,z_p$ with the canonical $GL_p$-action.
We extend the action of the pure and mixed trace algebras by
\[
g(\text{tr}(z_{i_1}\cdots z_{i_n})z_{j_1}\cdots z_{j_m})=\text{tr}(g(z_{i_1}\cdots z_{i_n}))g(z_{j_1}\cdots z_{j_m}),
\]
$z_{i_1}\cdots z_{i_n},z_{j_1}\cdots z_{j_m}\in R_{2p}$, $g\in GL_p$.
For a $p$-dimensional $GL_d$-module $W$, we consider the induced $GL_d$-action on $R_{2p}, C_{2p}$ and $T_{2p}$.
Let $d=2$. Then $W=W(2)\oplus W(0)$ is a 4-dimensional $GL_2$-module with Hilbert series
\[
H(W;x_1,x_2)=x_1^2+x_1x_2+x_2^2+1.
\]
The Hilbert series of $T_{24}$ is
\begin{eqnarray*}
H(T_{24};x_1,x_2,x_3,x_4)&=&\prod_{i=1}^4\frac{1}{(1-x_i)^2}\sum_{n\geq 0}S_{(n,n)}(X)
\\
&=&(1-x_1x_2x_3x_4)\prod_{i=1}^4\frac{1}{(1-x_i)^2}\prod_{1\leq i<j\leq 4}\frac{1}{1-x_ix_j}.
\end{eqnarray*}
Hence the Hilbert series of the $\mathbb Z$-graded $GL_2$-module $T_{24}$ is
\begin{eqnarray*}
f(x_1,x_2,t)&=&H(T_{24};x_1^2t,x_1x_2t,x_2^2t,t)
\\
&=&\frac{1-x_1^3x_2^3t^4}{((1-t)(1-x_1^2t)(1-x_1x_2t)(1-x_2^2t))^2}
\\
&&\mbox{}\times\frac{1}{(1-x_1^2t^2)(1-x_1x_2t^2)(1-x_2^2t^2)(1-x_1^3x_2t^2)(1-x_1^2x_2^2t^2)(1-x_1x_2^3t^2)}.
\end{eqnarray*}
Computing the multiplicity series $M'(f;v_1,v_2,t)$ and replacing $(v_1,v_2)$ with $(0,1)$ and $(1,1)$
we obtain, respectively, the Hilbert series of $T_{24}^{SL_2}$ and $T_{24}^{UT_2}$:
\begin{eqnarray*}
H(T_{24}^{SL_2};t)&=&\frac{1-t+t^2+2t^4+t^6-t^7+t^8}{(1-t)^3(1-t^2)^2(1-t^3)^3(1-t^4)^2},
\\
H(T_{24}^{UT_2};t)&=&\frac{(1-t+t^2)(1+3t^2+4t^3+6t^4+4t^5+3t^6+t^8)}{(1-t)^5(1-t^2)^2(1-t^3)^3(1-t^4)^2}.
\end{eqnarray*}
By considering the three-dimensional $GL_2$-module $W(2)$ and the induced $GL_2$-action on $T_{23}$,  we obtain
\[
H(T_{23};x_1,x_2,x_3)
=\frac{1}{(1-x_1)^2(1-x_2)^2(1-x_3)^2(1-x_1x_2)(1-x_1x_3)(1-x_2x_3)},
\]
\[
f(x_1,x_2,t)=H(T_{23};x_1^2t,x_1x_2t,x_2^2t)
\]
\[
=\frac{1}{((1-x_1^2t)(1-x_1x_2t)(1-x_2^2t))^2(1-x_1^3x_2t^2)(1-x_1^2x_2^2t^2)(1-x_1x_2^3t^2)},
\]
\[
H(T_{23}^{SL_2};t)=\frac{1+t^4}{(1-t^2)^3(1-t^3)^2(1-t^4)},
\]
\[
H(T_{23}^{UT_2};t)=\frac{1+2t^2+2t^3+2t^4+t^6}{(1-t)^2(1-t^2)^3(1-t^3)^2(1-t^4)}.
\]
\end{example}

As in the commutative case $K[Y]$ one may consider linear locally nilpotent derivations of
the free algebra $K\langle Y\rangle$ and of any relatively free algebra $F(R)$. Again, we call
such derivations Weitzenb\"ock derivations. There is a very simple condition when the algebra of
constants $F(R)^{\delta}$ of a nonzero Weitzenb\"ock derivation $\delta$ is finitely generated.
By a result of Drensky and Gupta \cite{DGu} if the T-ideal $T(R)$ of the polynomial identities of $R$
is contained in the T-ideal $T(U_2(K))$, then $F(R)^{\delta}$ is not finitely generated. The main result
of Drensky \cite{D4} states that if $T(R)$ is not contained in $T(U_2(K))$, then $F(R)^{\delta}$ is
finitely generated. For various properties and applications of Weitzenb\"ock derivations acting
on free and relatively free algebras see \cite{DGu}. The following theorem and its corollary combine Theorems
\ref{Weitzenboeck} and \ref{noncommutative SL- and UT-invariants}. We omit the proofs which repeat
the main steps of the proofs of these two theorems.

\begin{theorem}\label{noncommutative Weitzenboeck}
Let $\delta$ be a Weitzenb\"ock derivation of the relatively free algebra $F(R)$
with Jordan normal form consisting of $k$ cells
of size $d_1+1,\ldots,d_k+1$, respectively. Let
\[
f_{\delta}(x_1,x_2,t)=H(F(R);x_1^{d_1}t,x_1^{d_1-1}x_2t,\ldots,x_2^{d_1}t,\ldots,x_1^{d_k}t,\ldots,x_1x_2^{d_k-1}t,x_2^{d_k}t)
\]
be the function obtained from the Hilbert series of $F(R)$ by substitution
of the first group of $d_1+1$ variables $x_1,x_2,\ldots,x_{d_1+1}$
with $x_1^{d_1}t,x_1^{d_1-1}x_2t,\ldots,x_2^{d_k}t$,
the second group of $d_2+1$ variables $x_{d_1+2},x_{d_1+3},\ldots,x_{d_1+d_2+2}$ with
$x_1^{d_2}t,x_1^{d_2-1}x_2t,\ldots,x_2^{d_2}t$, $\ldots$, the $k$-th group of $d_k+1$ variables
$x_{d-d_k},\ldots,x_{d-1},x_d$ with $x_1^{d_k}t,\ldots,x_1x_2^{d_k-1}t,x_2^{d_k}t$.
Then the Hilbert series of the algebra of constants $F(R)^{\delta}$
is given by
\[
H(F(R)^{\delta};t)=M(f_{\delta};1,1),
\]
where $M(f_{\delta};x_1,x_2)$ is the multiplicity series of the symmetric with respect to $x_1,x_2$
function $f_{\delta}(x_1,x_2,t)\in K(t)[[x_1,x_2]]^{S_2}$.
Hence the Hilbert series $H(F(R)^{\delta};t)$ is a nice rational function.
\end{theorem}

\begin{corollary}\label{corollary of noncommutative Weitzenboeck}
Let $\delta$ be a Weitzenb\"ock derivation of the relatively free algebra $F(R)$
with Jordan normal form consisting of $k$ cells
of size $d_1+1,\ldots,d_k+1$, respectively. Let us identify the vector space $KY$ spanned by the free generators
of $F(R)$ with the $GL_2$-module
\[
W=W(d_1)\oplus \cdots\oplus W(d_k).
\]
Then the Hilbert series $H(F(R)^{\delta};t)$ and $H(F(R)^{UT_2};t)$ of the algebras of constants $F(R)^{\delta}$ and
of $UT_2$-invariants coincide.
\end{corollary}

\begin{example}\label{metabelian constants}
By Example \ref{metabelian algebras} the Hilbert series of the relatively free algebra $F(U_2(K))$ is
\[
H(F(U_2(K));X)=2\prod_{i=1}^d\frac{1}{1-x_i}+((x_1+\cdots+x_d)-1)\prod_{i=1}^d\frac{1}{(1-x_i)^2}.
\]
Let $d=3$ and let $\delta$ be the Weitzenb\"ock derivation with one three-dimensional cell acting on $F(U_2(K))$.
Following the procedure of Theorem \ref{noncommutative Weitzenboeck}, we define the function
\[
f(x_1,x_2,t)=\frac{2}{(1-x_1^2t)(1-x_1x_2t)(1-x_2^2t)}+\frac{(x_1^2+x_1x_2+x_2^2)t-1}{(1-x_1^2t)^2(1-x_1x_2t)^2(1-x_2^2t)^2}.
\]
As in Example \ref{nonlinear metabelian} we compute
\[
M'(f;v_1,v_2)=\frac{1-(v_1^2+v_2)t+(2v_1^2-v_2)v_2t^2+2(v_1^2+v_2)v_2^2t^3-2v_1^2v_2^3t^4}{(1-v_1^2t)^2(1-v_2t)(1-v_2^2t^2)^2},
\]
\[
H(F_3(U_2(K))^{\delta};t)=M'(f;1,1)=\frac{1-2t+t^2+4t^3-2t^4}{(1-t)^3(1-t^2)^2}.
\]
If $d=4$ and $\delta$ is a Weitzenb\"ock derivation with two $2\times 2$ cells, then
\[
f(x_1,x_2,t)=\frac{2}{(1-x_1t)^2(1-x_2t)^2}+\frac{2(x_1+x_2)t-1}{(1-x_1t)^4(1-x_2t)^4},
\]
\[
H(F_4(U_2(K))^{\delta};t)=\frac{1+10t^3+23t^4+2t^5-8t^6+2t^8}{(1-t)^2(1-t^2)^5}.
\]
\end{example}

\begin{example}\label{constants of trace algebras}
As in the case of invariants we can extend the derivations of the generic trace algebra $R_{kp}$ to the pure and mixed
trace algebras $C_{kp}$ and $T_{kp}$.
Let $\delta_{20}$ be the Weitzenb\"ock derivation with a three-dimensional and a one-dimensional Jordan cell
acting on the mixed trace algebra $T_{23}$.
Then Corollary \ref{corollary of noncommutative Weitzenboeck} and Example \ref{invariants of trace algebras}
give that
\[
H(T_{24}^{\delta};t)=H(T_{24}^{UT_2};t)=\frac{(1-t+t^2)(1+3t^2+4t^3+6t^4+4t^5+3t^6+t^8)}{(1-t)^5(1-t^2)^2(1-t^3)^3(1-t^4)^2}.
\]
If $\delta$ has one three-dimensional cell only, then again Example \ref{invariants of trace algebras} gives
\[
H(T_{23}^{\delta};t)=H(T_{23}^{UT_2};t)=\frac{1+2t^2+2t^3+2t^4+t^6}{(1-t)^2(1-t^2)^3(1-t^3)^2(1-t^4)}.
\]
\end{example}

\section*{Acknowledgements}
This project was started when the third named author visited the fifth named at the Department of Mathematics
of the North Carolina State University at Raleigh.
He is grateful to the Department
for the warm hospitality and the creative atmosphere.
He is also very thankful to Leonid Bedratyuk
for the stimulating discussions and suggestions on classical invariant theory
in Section \ref{section on classical invariant theory}.

\end{document}